\documentclass[12pt]{article}

\usepackage[utf8]{inputenc}
\usepackage[T1]{fontenc}

\usepackage[french,english]{babel}
\usepackage{enumerate} 
\usepackage{cancel}

\usepackage[svgnames]{xcolor}
\usepackage[pdftex,leftbars,xcolor]{changebar}

\usepackage{tikz}
\usetikzlibrary{matrix,arrows}

\usepackage{amsmath,amsthm}
\usepackage{amssymb,latexsym,amsfonts}

\usepackage{textcomp}

\allowdisplaybreaks[2]

\usepackage{aliascnt}
\numberwithin{equation}{section}

\usepackage{mathtools}
\usepackage{relsize,exscale}
\usepackage{rotating}
\usepackage{etoolbox}
\usepackage{etex}

\pdfoutput=1
\usepackage[pdftex,%
bookmarks=true,%
bookmarksnumbered=true,%
plainpages=false,%
breaklinks=true,%
colorlinks=false,%
urlcolor=blue,%
citecolor=green,%
linkcolor=red,%
]{hyperref}
\usepackage{doi}

\usepackage[hyperpageref]{backref}
\backreffrench
\renewcommand*{\backref}[1]{}  
\renewcommand*{\backrefalt}[4]{
  \ifcase #1 %
  \relax
  \or
(Cited page~#2.)%
  \else
(Cited pages~#2.)%
  \fi}

\newtheorem{theorem}{Theorem}[section]

\newaliascnt{cor}{theorem}
\newtheorem{cor}[cor]{Corollary}
\aliascntresetthe{cor}

\newaliascnt{prop}{theorem}
\newtheorem{prop}[prop]{Proposition}
\aliascntresetthe{prop}

\newaliascnt{lemma}{theorem}
\newtheorem{lemma}[lemma]{Lemma}
\aliascntresetthe{lemma}

\theoremstyle{definition}
\newaliascnt{defi}{theorem}
\newtheorem{defi}[defi]{Definition}
\aliascntresetthe{defi}

\newaliascnt{example}{theorem}

\aliascntresetthe{example}

\theoremstyle{remark}
\newaliascnt{remark}{theorem}
\newtheorem{remark}[remark]{Remark}
\aliascntresetthe{remark}

\addto\extrasenglish{%
}

\newcommand{\ie}{\textit{i.e.} \/}

\newcommand{\numberset}[1]{\mathbb{#1}}

\newcommand{\nat}{\numberset{N}}
\newcommand{\intg}{\numberset{Z}}

\newcommand{\korps}{\numberset{K}}

\DeclareMathOperator{\Kr}{Ker}
\DeclareMathOperator{\Der}{Der}

\DeclareMathOperator{\Hom}{Hom}
\DeclareMathOperator{\pr}{\mathsf{pr}}

\newcommand{\Sym}[1][\!]{\ensuremath{\mathcal{S}^{#1}}}
\newcommand{\Tens}[1][\!]{\ensuremath{\mathcal{T}^{#1}}}

\DeclareMathOperator{\UEA}{\mathcal{U}\!}

\DeclareMathOperator{\un}{\mathbf{1}}

\newcommand{\Hochcochains}[1][]{\ensuremath{C_{H}^{#1}}}
\newcommand{\Hoch}[1][]{\ensuremath{H_{H}^{#1}}}
\newcommand{\ChEcochains}[1][]{\ensuremath{C_{CE}^{#1}}}
\newcommand{\ChE}[1][]{\ensuremath{H_{CE}^{#1}}}

\begin{document}

\begin{center}
	\LARGE {$L_\infty$-Formality check for \\ the
		Hochschild Complex of certain \\ Universal Enveloping Algebras}
	
	\vspace*{1em}

	\Large
	{
		
		Martin Bordemann\up{\,1},%
		{\renewcommand{\thefootnote}{\fnsymbol{footnote}}
			Olivier Elchinger\up{\,2}\footnote{This author has been fully supported in the frame of the AFR scheme of the Fonds National de la Recherche (FNR), Luxembourg with the project QUHACO 8969106} \\
		}
		\setcounter{footnote}{0}
		Simone Gutt\up{\,3}, and Abdenacer Makhlouf\up{\,1}
	}
	
	\vspace*{2em}
	
	\normalsize
	{
		\up{1} %
		Laboratoire de Math\'{e}matiques, Informatique  et Applications, \\
		Universit\'{e} de Haute-Alsace, Mulhouse,\\
		\texttt{Martin.Bordemann@uha.fr,~Abdenacer.Makhlouf@uha.fr}
		
		\up{2}
		Mathematics Research Unit,
		University of Luxembourg,\\
		\texttt{Olivier.Elchinger@uha.fr}
		
		\up{3}
		D\'{e}partement de Math\'{e}matiques,
		Universit\'{e} Libre de Bruxelles\\
		\texttt{sgutt@ulb.ac.be}
	}
	
	\vspace*{1.5em}
	
	\today
\end{center}

\begin{abstract}
We study the $L_\infty$-formality problem for the Hochschild complex of the universal enveloping algebra of some examples of Lie algebras such as Cartan-$3$-regular quadratic Lie algebras (for example semisimple Lie algebras and in more detail $\mathfrak{so}(3)$), and free Lie algebras. We show that for these examples formality in Kontsevich's sense does NOT hold, but we compute the $L_\infty$ structure on the cohomology given by homotopy transfer in certain cases.
\end{abstract}

\section*{Introduction}

Since Maxim Kontsevich's seminal paper \cite{K03} on deformation quantization on any Poisson manifold, his concept of $L_\infty$-formality of the Hochschild complex of an associative algebra with values in the algebra turned out to be extremely useful for the deformation theory of that algebra. More precisely  the Hochschild complex seen as a differential graded Lie algebra (by means of the Hochschild
differential and the Gerstenhaber bracket) is called
formal if it is quasi-isomorphic in the $L_\infty$-sense to its Hochschild cohomology. If this is the case, first order deformations (seen as $2$-cocycles) having induced Gerstenhaber bracket equal to zero (so-called Maurer-Cartan elements) always integrate to formal deformations.

Kontsevich's basic example is the symmetric algebra of a finite dimensional vector space (over a field $\korps$ of characteristic $0$) whose Hochschild complex he showed is formal. An interesting playground for formality checks seems to be the class of \emph{universal enveloping algebras of Lie algebras} which are very close to symmetric algebras. Two of us
(M.B. and A.M.) have already looked at the Lie algebra of 
all infinitesimal affine transformations of $\korps^n$
where we found formality, see \cite{BM08}. One of us
(O.E.) has studied the three-dimensional Heisenberg algebra
and found that the corresponding Hochschild complex was NOT
formal, see \cite{El12} and \cite{E14}.\\
 The aim of the present work is to check formality of the Hochschild complex of
the universal enveloping algebras of two classes of Lie algebras: on one hand some finite-dimensional quadratic Lie algebras which we call Cartan-$3$-regular (e.g.~semisimple Lie algebras
and in more detail $\mathfrak{so}(3)$), and on the other hand free Lie algebras over any vector space. The second aim is
--if possible-- 
to explicitly compute
the higher brackets of order $\geqslant 3$ on the cohomology
which then will ensure an
$L_\infty$-quasi-isomorphism with the Hochschild complex
by homotopy transfer.
\\
Our first main result is that the Hochschild complex of the universal enveloping
algebra of a \emph{nonabelian reductive Lie algebra} is
NOT formal. In fact, we show a more general result for \emph{Cartan-$3$-regular quadratic Lie algebras} which are quadratic Lie algebras whose Cartan $3$-cocycle defines a nontrivial cohomology class. 
Moreover, in the case of $\mathfrak{so}(3)$, one just has to add one higher
bracket $d_3$ of order $3$ to restore the 
$L_\infty$-quasi-isomorphism with the Hochschild complex which we can describe explicitly.\\ 
The second main result consists in showing that the Hochschild complex of the universal enveloping
algebra of any \emph{free Lie algebra}
generated by a vector space of dimension $\geqslant 2$ is
NOT formal by explicit computations. Again by adding one higher order bracket $d_3$ of order $3$ we can restore 
the $L_\infty$-quasi-isomorphism with the Hochschild complex.\\
On the other hand note that the
universal enveloping algebras of semisimple and free Lie algebras
are well-known to be rigid, hence every
first order deformation integrates to a deformation which is equivalent
to the trivial deformation: it follows that these associative algebras
provide examples where the deformation problem can always be solved, but where formality does not hold.

The main tool for formality checks is a \emph{characteristic
$3$-class $c_3$} in the graded Chevalley-Eilenberg cohomology
of the graded Lie algebra given by the Hochschild cohomology
equipped with a graded Lie bracket induced by the Gerstenhaber Lie bracket: this is well-known in the litterature in order-by-order computations, and provides
the first obstruction to $L_\infty$-formality.\\
In order to deal with finite-dimensional Lie algebras we use a
result already sketched in \cite[Secs.~8.3.1,8.3.2]{K03} and further explicited
in \cite{BM08} (based on the work \cite{BMP05}) that the Hochschild complex of the universal enveloping algebra of a finite-dimensional Lie algebra $\mathfrak{g}$ as a 
differential graded
Lie algebra is quasi-isomorphic (in the $L_\infty$-sense)
to the much `easier' Chevalley-Eilenberg
complex of $\mathfrak{g}$ with values in
the symmetric algebra $\Sym \mathfrak{g}$ of the adjoint module
$\mathfrak{g}$ equipped with the Chevalley-Eilenberg differential and the Schouten bracket for
poly-vector fields (on the manifold $\mathfrak{g}^*$).
The proof we know of involves Kontsevich's formality map.
For quadratic Lie algebras, we compute the Schouten brackets of polynomials
in the quadratic Casimir (degree $0$) with the Cartan $3$-cocycle (degree $3$) where the Lie bracket (degree $2$)
and the Euler field (degree $1$)
appear, and evaluate a representing graded $3$-cocycle for
the $3$-class $c_3$ on combinations of these classes.
\\
For the free Lie algebra generated by a vector space
$V$ we use the classical result that
its universal enveloping algebra is simply isomorphic to the free
algebra $\Tens V$ generated by $V$ whose Hochschild cohomology is also well-known to be concentrated in degree $0$ and $1$.
Here one of the main tools is to construct an explicit complement to the space of all inner derivations inside the
space of all derivations in case $V$ is finite-dimensional.\\
Another general tool which we shall use several times is the 
homotopy perturbation lemma (in its
$L_\infty$-form) to compute the higher order
brackets if necessary.

The paper is organized as follows: \autoref{Sec:Kontsevich-formality} recalls graded coalgebraic structures, definitions of $L_\infty$ algebras and morphisms, and of formality, and the characteristic $3$-class.
\autoref{Sec:Perturbation-lemma} presents the \emph{Perturbation Lemma} for chain complexes and its well-known extension to $L_\infty$-algebras. We add a seemingly less known observation, see \autoref{Thm:PertBordemann}
(for which the proof will be in \cite{BE18}), that
the ordinary geometric series formulas from the `unstructured'
perturbation lemma automatically preserve the underlying
graded coalgebra structures. In particular, this result allows to transfer the $L_\infty$-structure of the Hochschild cochain complex to its cohomology and allows to construct quasi-inverses.
In \autoref{Sec:Finite-dim-Lie-alg} and \autoref{Sec:Free-Lie-alg}, we show that the characteristic
$3$-class cannot be zero for certain examples by computing
a representing $3$-cocyle on well-chosen elements.
Moreover, by means of
the perturbation lemma we show that for the Lie algebra $\mathfrak{so}(3)$, 
 and for free Lie algebras, the $L_\infty$ structure on the cohomology does not come from the Gerstenhaber bracket only, but involves a computable map of arity $3$.

\subsubsection*{Acknowledgements} The authors would like to
thank B.~Hurle,~B.~Valette,~F.~Wagemann, and P.~Xu for fruitful discussions, and T.~Petit for making us aware of
Sections 8.3.1 and 8.3.2 in Kontsevich's article \cite{K03}.

\section{Kontsevich formality} \label{Sec:Kontsevich-formality}

\subsection{Generalities}\label{SubSec:Generalities}

The material of the following Section is mostly contained
in \cite[p.40-50]{FHT01}, \cite{K03}, \cite{AMM02}, the Appendix of \cite{BGHHW05}, \cite[§4.1]{BM08}, 
\cite{LV12}, \cite{El12}, \cite{Bor15}).

Let $\korps$ be a field of characteristic $0$. We will use the framework of graded bialgebras. Unless explicitly specified, vector spaces $V,W,\ldots$ and algebras will be graded over $\intg$,
and the degree in $\mathbb{Z}$ of a homogeneous element $x$
will be denoted by $|x|$. $\Hom (V,W)$ will always denote
the subspace of the space of all linear maps $V\to W$ generated by all homogeneous linear maps $V\to W$. Tensor products
of graded vector spaces $V$ and $W$ are graded as usual.

We recall the notations used: there is the
graded transposition $\tau:V\otimes W\to W\otimes V$
defined by $\tau(x\otimes y)=(-1)^{|x||y|}y\otimes x$
on homogeneous elements $x\in V$ and $y\in W$, and the
\begin{align*}
\text{Koszul rule of signs} & & (\phi \otimes \psi)(x \otimes y) \coloneqq (-1)^{|\psi| |x|} \phi(x) \otimes \psi(y),
\intertext{for homogeneous linear maps $\phi$ and $\psi$ between graded spaces; and in the tensor product of two
	graded algebras $\mathcal{A}$ and $\mathcal{B}$ the}
\text{graded multiplication} & & (a \otimes b)(a' \otimes b') \coloneqq (-1)^{|b||a'|} aa' \otimes bb',
\end{align*}
for elements $a,a'\in\mathcal{A}$ and $b,b'\in\mathcal{B}$.

\emph{Shifted graded vector spaces} are noted $V[j]$ for any integer $j$, with $V[j]^i \coloneqq V^{i+j}$ for all integers $i$. The \emph{suspension map} $s : V \to V[-1]$, defined by the identity of the underlying vector spaces, is of degree one. Multilinear maps $\phi : V^{\otimes k}\to W^{\otimes l}$ can be shifted, \ie
$\phi[j] : V[j]^{\otimes k} \to W[j]^{\otimes l}$ by setting $\phi[j] \coloneqq (s^{\otimes l})^{-j} \circ \phi \circ 
(s^{\otimes k})^{j}$. Note that $(s^{\otimes k})^{j}
=(-1)^{\frac{j(j-1)}{2}\frac{k(k-1)}{2}}(s^j)^{\otimes k}$.
The degree of the shifted map $\phi[j]$ is given by $|\phi[j]| = j(k-l) + |\phi|$. For $k=l$, the maps $\phi$ and
its shift $\phi[j]$ have the same degree and the same action on
the underlying ungraded vector spaces. We clearly have
the following rules
\begin{equation}
  \phi[j][j']=\phi[j+j'],~~
  (\psi\circ\phi)[j]=(\psi[j])\circ (\phi[j]),~~\mathrm{and}
  \label{EqDefShiftJPlusJPrimeAndCirc}
\end{equation}
\begin{equation}
  (\phi\otimes \phi')[j]=
  (-1)^{\frac{j(j-1)}{2}kk'+\frac{j(j+1)}{2}ll'
  	     +j(kl'+k|\phi'|+|\phi|l')}
  	   (\phi[j])\otimes (\phi'[j]).
  	 \label{EqDefShiftTensorProduct}  
\end{equation}
where $j,j'$ are integers and $\phi'$ is a homogeneous $\korps$-linear map from $V^{\prime \otimes k'}$
to $W^{\prime \otimes l'}$, $V'$ and $W'$ being graded
vector spaces.

Recall the \emph{graded symmetric bialgebra of $V$} 
$(\Sym V,\mu_{sh}=\bullet,\Delta_{sh},\un,\varepsilon)$, with $\mu_{sh}$ the graded commutative multiplication and $\Delta_{sh}$ the graded cocommutative (shuffle) comultiplication:
it is defined by the free algebra $\Tens V$ modulo the
two-sided graded ideal generated by
$x\otimes y-(-1)^{|x||y|}y\otimes x$ for any homogeneous elements $x,y$ in $V$. Note that we shall keep the notation
$\Lambda V$ (used in the framework of rational homotopy theory
for the graded symmetric algebra,
see e.g.~\cite{FHT01}) for another object, the graded Grassmann algebra which will be explained further down. $\Sym V$ is well-known to be graded cocommutative connected as a coalgebra, the canonical filtration 
(see Appendix \ref{App: Graded Coalgebras})
simply being given by  $\oplus_{s=0}^r\Sym[s] V$ for all $r\in\mathbb{N}$. Moreover,
$\Sym V$ is \emph{free} as a graded commutative
associative  algebra, and \emph{cofree} as a graded cocommutative
coassociative connected coalgebra. This means that
for any graded cocommutative
coassociative connected coalgebra
$(C,\Delta_C,\varepsilon_C, \un_C)$ and
any linear map of degree $0$, $\varphi : C\to V$ such that
$\varphi(\un_C)=0$ there is a unique morphism of graded connected coalgebras $\Phi : C \to \Sym V$
such that $\pr_V\circ \Phi=\varphi$ where 
$\pr_V:\Sym V\to V$ denotes the obvious projection. Conversely,
every such morphism is given this way. Likewise, for any given morphism of graded connected coalgebras $\Psi:C \to\Sym V$, 
and any linear map $d : C \to V$ there is a unique graded coderivation of graded counital coalgebras $\mathcal{D}:C \to \Sym V$ along $\Psi$ such that $\pr_V\circ \mathcal{D}=d$. Again
every graded coderivation of unital coalgebras 
$\mathcal{D}:C \to \Sym V$ along $\Psi$ is given that way.  These induced maps can be computed 
as $\Phi = e^{*\varphi}$ 
(which we shall use henceforth as notation), and 
$\mathcal{D} = d * \Psi$ (for which we shall use the abridged notation $\overline{d}$) where $*$ is the corresponding \emph{convolution multiplication} in 
$\Hom(C,\Sym V)$
(see Appendix \ref{App: Graded Coalgebras}) for the definition
of convolution, and 
see e.g.~\cite{Hel89} (or \cite[App.A]{Bor15}) for the given formula.  The convolution
exponential
converges since $\varphi$ is of filtration degree $-1$ if
$C$ carries the canonical filtration and
$\Sym V$ is equipped with the trivial discrete filtration.
In the case $C=\Sym U$ which we shall mostly encounter in this
paper the maps $\mathcal{D}$ and $e^{*\varphi}$ are uniquely determined by the sequence of restrictions $d_n \coloneqq d|_{\Sym[n] U}\to V$
and $\varphi_n\coloneqq \varphi|_{\Sym[n] U}\to V$ (also
called \emph{Taylor coefficients}). 
It is not entirely necessary for this paper, but we would like
to mention that the graded commutator 
$[\overline{d_1},\overline{d_2}]$ of two coderivations
$\overline{d_1}, \overline{d_2}$ of $\Sym V$ (along the identity) is again a graded coderivation of $\Sym V$ (along the identity)
which is induced by the so-called
\emph{Nijenhuis-Richardson multiplication} $\circ_{NR}$ 
and \emph{Nijenhuis-Richardson bracket}
$[~,~]_{NR}$
on the space $\Hom(\Sym V,V)$ defined by
\[
  d_1\circ_{NR} d_2=d_1\circ \overline{d_2},~~\text{and}~~
  [d_1,d_2]_{NR}=d_1\circ_{NR} d_2 -(-1)^{|d_1||d_2|}d_2\circ_{NR} d_1
\]
whence $\big[\overline{d_1},\overline{d_2}\big]=
\overline{[d_1,d_2]_{NR}}$. The Nijenhuis-Richardson bracket
is well-known to be a graded Lie bracket, and the Nijenhuis-Richardson multiplication is a non-associative multiplication called graded pre-Lie, see \cite{Ger63}.

We shall also have to use the \emph{graded exterior algebra},  $\Lambda V$, defined as the free algebra $\Tens V$
modulo the two sided ideal generated by
$x\otimes y+(-1)^{|x||y|}y\otimes x$ for any homogeneous
elements $x,y$ in $V$. The induced multiplication is denoted by $\wedge$. The exterior algebra is $\mathbb{Z}\times \mathbb{Z}$-graded
(or \textbf{bigraded}) where an element 
$x_1\wedge\cdots \wedge x_k$
(where $x_1,\ldots,x_k$ are homogenous elements of $V$)
has bidegree $(k,|x_1|+\cdots+|x_k|)$. Assigning to
a pair of bidegrees $(k,i)$ and $(l,j)$ the grading
sign $(-1)^{kl+ij}$ the graded Grassmann algebra becomes a bigraded (co)commutative bialgebra. Note that for the comultiplication to make sense a bigraded Koszul rule
has to be used.
Recall that the shift $\phi\mapsto \phi[j]$ of multilinear maps
switches from graded symmetric to graded exterior algebras
for odd $j$: if $\varphi$ is a homogenous linear map
from $\Sym[k]V \to V'$ then its shift $\phi=\varphi[j]$
can be viewed as a bihomogeneous linear map
$\Lambda^k (V[j])\to V'[j]$ of bidegree $(1-k,|\varphi|+(k-1)j)$ if $j$ is odd.
In particular, a graded antisymmetric bilinear map
$c:\Lambda^2 W\to W$ of bidegree $(-1,0)$ will shift 
to a graded symmetric bilinear map $c[1]:\Sym[2] W[1]\to W[1]$
of degree $1$.

We note the following well-known
\begin{lemma} \label{Lem:MorphCoInv}
Let $\Phi\coloneqq e^{*\varphi} : \Sym U\to \Sym V$ 
be a morphism of graded
connected coalgebras. Then it is a bijection if and only if
the component $\varphi_1=e^{*\varphi}|_{U}:U\to V$ is a $\korps$-linear bijection.
\end{lemma}
\begin{proof}
	If $\Phi$ is bijective with inverse $\Psi$ (which is also a map of graded connected coalgebras) then both maps
	are filtration preserving, hence $\Phi(U)\subset V$ and
	$\Psi(V)\subset U$ and hence induce $\korps$-linear
	bijections $\Phi|_U=\varphi_1:U\to V$ and 
	$\Psi|_V=\psi_1:V\to U$, respectively.\\
	Conversely, suppose that $\Phi|_U=\varphi_1:U\to V$ is bijective with
	inverse $\psi_1:V\to U$. By evaluating the projections
	$\pr_U$ and $\pr_V$, respectively, it follows that
	$e^{*\psi_1}\circ e^{*\varphi_1}=id_{\Sym U}$, and
	$e^{*\varphi_1}\circ e^{*\psi_1}=id_{\Sym V}$
	(equation of maps of graded connected coalgebras).
	Set $\varphi = \varphi_1 + \varphi'$ with 
	$\varphi' = \sum_{k=2}^\infty \varphi_k$ where $\varphi_k : \Sym[k] U \to V$. We can write
	\begin{equation*}
	 e^{* \varphi} = e^{* (\varphi_1 + \varphi')}
	= e^{* \varphi_1} * e^{* \varphi'}
	= e^{* \varphi_1}+e^{*\varphi_1} * \left( e^{* \varphi'} - \un \varepsilon \right)=: 
	e^{* \varphi_1}\circ (id_{\Sym U}+ A),
	\end{equation*}
	where $A=e^{*\psi_1}\circ\big(e^{*\varphi_1}*(e^{* \varphi'} - \un \varepsilon)\big)$ clearly is of filtration degree $-1$. Hence the inverse of $id_{\Sym U}+ A$ is given
	by the geometric series $\sum_{r=0}^\infty (-A)^{\circ r}$
	which converges by the completeness of 
	$\Hom (\Sym U,\Sym U)$, see Appendix \ref{App: Filtered vector spaces}. It follows that $\Phi$ has an inverse 
	$\Psi= (id_{\Sym U}+A)^{-1}\circ e^{*\psi_1}$.
\end{proof}
\noindent  See also \cite[Sec.~10.4,~Thm.~10.4.1]{LV12} for the 
more general operadic version.
\bigskip

Let $W = \oplus_{j \in \intg} W^j$ a $\intg$-graded vector space. Let $V=W[1]$ be the shifted graded vector space.

\begin{defi} \label{Def:L-InftyStruct}
	A \emph{$L_\infty$-structure on $W$} is defined to be a graded coderivation $\mathcal{D}=\overline{D}$ of $\Sym\,(W[1])$ of degree $1$ satisfying $\mathcal{D}^2=0$ and $\mathcal{D}(\un_{\Sym W[1]})=0$. 
	The pair $(W,\mathcal{D})$ is called an \emph{$L_\infty$-algebra}.
\end{defi}
Note that $D=\sum_{r=1}^\infty D_r$ is a $\korps$-linear map
$\Sym (W[1]) \to W[1]$ of degree $1$, and the first component $\mathcal{D}|_{W[1]}=D_1:W[1]\to W[1]$ of $\mathcal{D}$ is a differential, \ie $D_1^2=0$. $L_\infty$-algebras with
$D_1=0$ are called \emph{minimal}. In this  paper
we shall encounter particular $L_\infty$-algebras for which
$D_n=0$ for all integers $n\geqslant n_0$ for some nonnegative integer $n_0$.

A \emph{$L_\infty$-morphism} from a $L_\infty$-algebra 
$(W,
\mathcal{D})$ to a $L_\infty$-algebra $(W',\mathcal{D}')$ is a morphism of differential graded connected coalgebras $\Phi : \big(\Sym(W[1]),\mathcal{D}\big) \to \big(\Sym(W'[1]),
\mathcal{D}'\big)$, \ie $\Phi=e^{*\varphi}$ is a morphism of graded connected coalgebras (see Appendix \ref{App: Graded Coalgebras}) intertwining differentials,
\begin{equation}\label{EqDefLinfinityMorphism}
   \Phi \circ \mathcal{D} = \mathcal{D}' \circ \Phi.
\end{equation}
Moreover, a $L_\infty$-map $\Phi$ is called an \emph{$L_\infty$-quasi-isomorphism} if its first component $\Phi_1= \Phi|_{W[1]}=\varphi_1:W[1]\to W'[1]$ --which is a chain map $(W[1],D_1)\to (W'[1],D'_1)$-- induces an isomorphism in cohomology. It can be shown that every 
$L_\infty$-quasi-isomorphism has a \emph{quasi-inverse}, \ie 
a $L_\infty$-morphism $\Psi=e^{*\psi}$ going from $\big(\Sym(W'[1]),\mathcal{D}'=\overline{D'}\big)$ to
$\big(\Sym(W[1]),\mathcal{D}=\overline{D}\big)$ such that the chain map $\psi_1=\Psi|_{W'[1]}:(W'[1],D'_1)
\to (W'[1],D'_1)$ induces an isomorphism in cohomology,
see e.g.~\cite{K03}, \cite[Prop.~V2]{AMM02}, or
\cite[Thm.~10.4.4]{LV12}. It even
follows that $\Phi$ and $\Psi$ induce isomorphism of the
cohomologies with respect to the entire differentials 
$\mathcal{D}$ and $\mathcal{D}'$ but we shall not need this statement.

A very important example of a $L_\infty$-algebra (motivating
the whole structure) is a
\emph{differential graded Lie algebra} $(\mathfrak{G},b,[~,~])$,
\ie $(\mathfrak{G},[~,~])$ is a graded Lie algebra and the
$\korps$-linear map $b:\mathfrak{G}\to\mathfrak{G}$ is of degree
$1$, $b^2=0$, and $b$ is a graded derivation of the 
graded Lie bracket $[~,~]$. In this case, on the shifted space $V=\mathfrak{G}[1]$, one sets $D_1=b[1]$, and $D_2=[~,~][1]$, $D_n=0$
for all $n\geqslant 3$, and the structure of a differential graded
Lie algebra ensures that $\overline{D}^2=0$. Hence
one gets a $L_\infty$-structure on $\mathfrak{G}$. Moreover,
it is well-known that its
cohomology $\mathfrak{H}$ with respect to $b$ carries a canonical graded Lie bracket $[~,~]_H$ induced from $[~,~]$.
Likewise, on the shifted space $\mathfrak{H}[1]$ the map $d=d_2=[~,~]_H[1]$ is the Taylor coefficient of order $2$ of a coderivation $\overline{d}$ of square
zero on $\Sym (\mathfrak{H}[1])$.\\
The \emph{formality problem for differential 
graded Lie algebras} is the question whether there is an
$L_\infty$-quasi-isomorphism $\Phi=e^{*\varphi}$
from $\big(\Sym (\mathfrak{H}[1]), \overline{[~,~]_H[1]}\big)$
to $\big(\Sym (\mathfrak{G}[1]),\overline{(b[1]+[~,~][1])}\big)$.
This is the analog of D.~Sullivan's formality for differential
graded associative algebras, see e.g.~\cite[p.156]{FHT01}.\\
We shall not need this in the sequel, but recall
that a stronger and more classical notion is a 
 \emph{quasi-isomorphism
of differential graded Lie algebras} which is a morphism of
differential graded Lie algebras whose induced morphism of
graded Lie algebras on cohomologies is an isomorphism.
Clearly, for a given quasi-isomorphism there is in general
not a quasi-isomorphism in the other direction inducing
the inverse morphism on cohomology. There is the notion of two
differential graded Lie algebra being \emph{weakly quasi-isomorphic}
if there is a finite zig-zag of quasi-isomorphisms of intermediate
differential graded Lie algebras with ends at the two given
Lie algebras. This notion turns out to be equivalent to
$L_\infty$-quasi-isomorphism, 
see e.g.~\cite[p.423, Thm.11.4.9]{LV12}.

There are several important differential graded Lie algebras 
describing the identities, the (co)homology, and the algebraic deformation theory of certain classes of algebras, see 
e.g.~the book \cite{LV12}. The one of interest in this paper 
concerns
the class of associative algebras and has been invented by M.~Gerstenhaber, \cite{Ger63}:\\
Let $(\mathcal{A},\mu)$ be an associative (not necessarily unital and trivially graded) algebra over the field $\korps$ of characteristic $0$. On the \emph{Hochschild complex} $\Hochcochains \coloneqq \bigoplus_{n \in \nat} \Hochcochains[n](\mathcal{A},\mathcal{A})$ of $\mathcal{A}$
with values in the bimodule $\mathcal{A}$ (considered with
its natural grading by number of arguments), recall the \emph{Gerstenhaber multiplication} 
$\circ_G : \Hochcochains \times \Hochcochains \to \Hochcochains$ which is a bilinear map of degree $-1$ defined for any nonnegative
integers $k,l$ and any $f\in \Hochcochains[k](\mathcal{A},\mathcal{A})$ and any
$g\in \Hochcochains[l](\mathcal{A},\mathcal{A})$
by (for any $a_1,\ldots,a_{k+l-1}\in\mathcal{A}$)
\begin{equation}\label{EqDefGerstenhaberMultiplication}
\begin{split}
(f \circ_G g)&(a_1,\dotsc ,a_{k+l-1}) = \\
& \sum_{i=1}^{k} (-1)^{(i-1)(l-1)} f(a_1,\dotsc ,a_{i-1},g(a_i,\dotsc ,a_{i+l-1}),a_{i+l},\dotsc ,a_{k+l-1}).
\end{split}
\end{equation}
This multiplication can be considered on the shifted space
$\mathfrak{G}(\mathcal{A})=\mathfrak{G} \coloneqq \Hochcochains{}[1]$
and the graded commutator,
\begin{equation}\label{EqDefGerstehaberBracket}
    [f,g]_G=f \circ_G g - (-1)^{(k-1)(l-1)} g \circ_G f,
\end{equation}
(where of course $k-1$ is the shifted degree of a $k$-cochain
$f$) is called the \emph{Gerstenhaber bracket} and turns out
to be a graded Lie bracket on $\mathfrak{G}$. The proof of
the graded Jacobi-identity is largely simplified by the 
classical observation that $\mathfrak{G}$ is isomorphic
to the space of all coderivations of the tensor algebra
$\Tens (\mathcal{A}[1])$ equipped with the usual deconcatenation
comultiplication (for which it is connected and cofree), but this is not important for the sequel.

Note that the algebra multiplication $\mu$
of $\mathcal{A}$ did not enter in the definition of the Gerstenhaber multiplication and bracket. For the shifted
version $\mathfrak{G}=\Hochcochains{}[1]$ any bilinear map $\mu:\mathcal{A}\times \mathcal{A}
\to \mathcal{A}$ is of degree $1$, and gives rise to an associative multiplication iff $[\mu,\mu]_G=0$. Moreover,
for any such $\mu$ the square of $b := [\mu,~]_G$ vanishes and defines, up
to a global sign, the Hochschild coboundary operator on
the complex $\Hochcochains{}[1]$. Hence $\big(\mathfrak{G}(\mathcal{A}),[~,~]_G,b = [\mu,~]_G\big)$
is a differential graded Lie algebra associated to any associative algebra $(\mathcal{A},\mu)$. It follows that
the shifted \emph{Hochschild cohomology}, $\mathfrak{H}(\mathcal{A}) =
\Hoch(\mathcal{A},\mathcal{A})[1]$ is a graded Lie algebra
with the induced Lie bracket $[~,~]_H$.\\
Maxim Kontsevich has linked the deformation problem in
deformation quantization of Poisson manifolds to the 
formality problem of the above differential graded Lie algebra
built on the Hochschild cochain complex of the
algebra $\mathcal{A}$ of all smooth functions on the underlying differentiable manifold, see \cite{K03}.

The problem we would like to consider in this paper
is contained in the framework of the following slight generalization of the formality
problem: let $(\mathfrak{G},[~,~],b)$ be a differential
graded Lie algebra, and let $(\mathfrak{H},[~,~]_H)$ be its
cohomology graded Lie algebra. Again let $D=D_2=b[1]+[~,~][1]$
denote the Taylor coefficients of the corresponding coderivation $\overline{D}$ of 
$\Sym (\mathfrak{G}[1])$, and we fix this $L_\infty$-structure
on $\mathfrak{G}$. We shall put a general minimal
$L_\infty$-structure on $\mathfrak{H}$ whose coderivation
$\overline{d}$ (of $\Sym (\mathfrak{H}[1])$) is given by a series
$d= d_2+\sum_{k \geqslant 3}d_k=d_2+d'$ where $d_2=[~,~]_H[1]$. It is well-known (see \cite{K03}, \cite{AMM02}, \cite{LV12})
that it is always possible to find a sequence
of `higher order brackets' $d_k:\Sym[k](\mathfrak{H}[1])
\to \mathfrak{H}[1]$ for $k \geqslant 3$ and a $L_\infty$-quasi-isomorphism $\Phi=e^{*\varphi}:
\big(\Sym (\mathfrak{H}[1]),\overline{[~,~]_H[1]+d'}\big)\to
\big(\Sym (\mathfrak{G}[1]),\overline{b[1]+[~,~][1]}\big)$
which also follows from the homotopy perturbation Lemma,
see the next Section. If all the higher order brackets $d'_k$ for $k\geqslant 3$
vanish there is formality, and if formality holds then
they can be transformed away by the conjugation with an invertible chain map
of differential graded coalgebras $\Sym (\mathfrak{H}[1])
\to \Sym (\mathfrak{H}[1])$.

\subsection{Low orders and a characteristic $3$-class}

This Subsection is well-known for the description of 
recursive obstructions to $L_\infty$-quis, see e.g.~
\cite{GH03}, Appendix of \cite{BGHHW05}, \cite{LV12}.
We have included it to get explicit formulae.
Let $\big(\mathfrak{G},[~,~]_G,b\big)$ be a differential
graded Lie algebra, and let 
$\big(\mathfrak{H},[~,~]_H,\big)$ be its cohomology. As before let $D=D_1+D_2=b[1]+ [~,~]_G[1]$ and $d=d_2=[~,~]_H[1]$ the 
Taylor coefficients of the corresponding graded coderivations
$\overline{D}$ and $\overline{d}$ of degree $1$
of $\Sym (\mathfrak{G}[1])$ and $\Sym (\mathfrak{H}[1])$, respectively. Clearly $\overline{D}^2=0$ and
$\overline{d}^2=0$.
Let $\varphi=\sum_{r=1}^{\infty}\varphi_r$ be a
$\korps$-linear map of degree $0$ 
from $\Sym (\mathfrak{H}[1])$ to $\Sym (\mathfrak{G}[1])$
with $\varphi_r=\varphi|_{{\Sym}^r (\mathfrak{H}[1])}$,
and let $\Phi=e^{*\varphi}$ the corresponding morphism
of connected coalgebras.
Define, as in \cite[Prop.A3]{BGHHW05},
\begin{equation}
  \overline{P}(\varphi)=\overline{D}\circ \Phi-\Phi\circ \overline{d},~~
  \mathrm{and}~~P(\varphi)=
  D\circ e^{*\varphi}-\varphi\circ \overline{d},
\end{equation}
the latter being the projection to $\mathfrak{G}[1]$. Clearly,
$\overline{P}(\varphi)$ is a graded coderivation of degree $1$ from
$\Sym (\mathfrak{H}[1])$ to $\Sym (\mathfrak{G}[1])$ along
$\Phi$, and is of course uniquely determined by its Taylor coefficient $P(\varphi)$. We shall call $\Phi$ an 
\emph{$L_\infty$-morphism of order $r$} if $P_s(\varphi)=0$
for all integers $1\leq s\leq r$. Clearly, $\Phi$ is an
$L_\infty$ morphism iff it is a $L_\infty$-morphism 
of order $r$ for each positive integer $r$. Since
$\Phi$ is filtration preserving it clearly follows that
$\Phi$ being a $L_\infty$-morphism of order $r$ gives only conditions on the maps $\varphi_1,\ldots,\varphi_r$, the higher orders not being affected.
We can `regauge'
$\Phi$ in the following way: let $\alpha:
\Sym (\mathfrak{H}[1])\to\mathfrak{H}[1]$ and
$\beta:\Sym (\mathfrak{G}[1])\to \mathfrak{G}[1])$ be $\korps$-linear maps of degree $0$ vanishing on the unit elements, and write $\mathcal{A}=e^{*\alpha}$ and
$\mathcal{B}=e^{*\beta}$ for the corresponding morphisms of
connected coalgebras. Supposing that $\mathcal{A}\circ \overline{d}=\overline{d}\circ \mathcal{A}$ and
$\mathcal{B}\circ \overline{D}=\overline{D}\circ \mathcal{B}$
it is straight-forward to see that the regauged $\varphi$,
\ie $\varphi'=\beta\circ \Phi\circ \mathcal{A}$, satisfies
\begin{equation}\label{EqCompGaugingP}
  P(\varphi')
  = \beta\circ \overline{P}(\varphi)\circ \mathcal{A},
\end{equation}
whence for each positive integer $r$ it follows that $e^{*\varphi'}= \mathcal{B}\circ \Phi\circ \mathcal{A}$
is a $L_\infty$-morphism of order $r$ if $\Phi$ is.

We compute $P_r(\varphi)$ for $r=1,2,3$: for any homogeneous
elements $x_1,x_2,x_3\in\mathfrak{H}[1]$
\begin{eqnarray}
   P_1(\varphi) & = &
         b[1]\circ \varphi_1,    \label{EqCompP1} \\
   P_2(x_1\bullet x_2)  & = &
        b[1]\big(\varphi_2(x_1\bullet x_2)\big)
        + D_2\big(\varphi_1(x_1)\bullet\varphi_1(x_2)\big)
               \nonumber \\
    &   &    -\varphi_1\big(d_2(x_1\bullet x_2)\big),
    \label{EqCompP2} \\
   P_3(x_1\bullet x_2\bullet x_3)  
       & = &
       b[1]\big(\varphi_3(x_1\bullet x_2\bullet x_3)\big)
       + D_2\big(\varphi_1(x_1)\bullet
             \varphi_2(x_2\bullet x_3)\big)\nonumber \\
     & &    + (-1)^{|x_1||x_2|}
         D_2\big(\varphi_1(x_2)\bullet
         \varphi_2(x_1\bullet x_3)\big)
         \nonumber \\
         & & +(-1)^{|x_3|(|x_1|+|x_2|)}
         D_2\big(\varphi_1(x_3)\bullet
         \varphi_2(x_1\bullet x_2)\big) \nonumber \\
         & & -\varphi_2\big(d_2(x_1\bullet x_2)\bullet x_3\big)
         -(-1)^{|x_2||x_3|}
            \varphi_2\big(d_2(x_1\bullet x_3)\bullet x_2\big) 
            \nonumber \\
            & &
            - (-1)^{|x_1|(|x_2|+|x_3|)}
            \varphi_2\big(d_2(x_2\bullet x_3)\bullet x_1\big). 
            \label{EqCompP3}
\end{eqnarray}
 Let $Z\mathfrak{G}$ and
$B\mathfrak{G}$ denote the graded vector space of all 
cocycles and coboundaries of the complex
$(\mathfrak{G},b)$, respectively, let 
$\pi:Z\mathfrak{G}\to \mathfrak{H}$ be the canonical projection
which obviously is a morphism of graded Lie algebras.
Then the cohomology of the shifted complex
$(\mathfrak{G}[1],b[1])$ can be identified with the shifted
cohomology $\mathfrak{H}[1]$. An $L_\infty$-morphism $\Phi=e^{*\varphi}:\Sym (\mathfrak{H}[1])\to
\Sym (\mathfrak{G}[1])$ of order $r$
is called a $L_\infty$-quis of order $r$ if 
$\varphi_1:\mathfrak{H}[1]\to\mathfrak{G}[1]$
induces an isomorphism in cohomology. The latter condition
is equivalent to stating that $b[1]\circ \varphi_1=0$ and
$\pi[1]\circ \varphi_1:\mathfrak{H}[1]\to\mathfrak{H}[1]$ is
invertible. More specifically, we shall call any linear map
map of degree $0$ 
$i:\mathfrak{H}[1]\to \mathfrak{G}[1]$ a \emph{section}
if all the values of $i$ are in $Z\mathfrak{G}[1]$ and if
\begin{equation}\label{EqDefSection}
\pi[1]\circ i= \mathrm{id}_{\mathfrak{H}[1]}.
\end{equation}
 The following
statements seem to be well-known:
\begin{lemma}\label{LP1P2Vanishing}
	With the above notations one has:
	\begin{enumerate}
		\item There is a $L_\infty$-quis 
		$\Phi=e^{*\varphi}:\Sym (\mathfrak{H}[1])\to
		\Sym (\mathfrak{G}[1])$
		of order $2$ such that $\varphi_1$ is a section.
  \item Let $\Phi=e^{*\varphi}$and $\Psi=e^{*\psi}$ be two $L_\infty$-quis of order $2$ where $\varphi_1$ is a section.\\
     Then there are linear maps $\alpha_1:\mathfrak{H}[1]\to
     \mathfrak{H}[1]$ and $\beta:\Sym (\mathfrak{G}[1])\to
     \mathfrak{G}[1]$ of degree $0$ (where $\beta$ vanishes on
     the unit), and $\chi_2:{\Sym}^2(\mathfrak{H}[1])\to
     \mathfrak{G}[1]$ of degree $0$ such that the morphisms of connected coalgebras
     $e^{*\alpha_1}$ and $e^{*\beta}$ commute with the corresponding differentials $\overline{d}$ and
     $\overline{D}$, respectively, and such that the regauged
     morphism 
     $e^{*\psi'}=\Psi'=e^{*\beta}\circ \Psi\circ e^{*\alpha_1}$
	 is a $L_\infty$-quis of order $2$ with
	 \begin{equation}\label{EqCompP1P2EqualZeroSectionCompare}
	   \psi'_1=\varphi_1,~~b[1]\circ \chi_2=0,
	     ~~\mathrm{and}~~
	   \psi_2'=\varphi_2+\chi_2.
	 \end{equation}	
	\item Let
	$\varphi:\Sym (\mathfrak{H}[1])\to \mathfrak{G}[1]$
	be a linear map of degree $0$ vanishing on $\un$.
	Then $\Phi=e^{*\varphi}$ is a $L_\infty$-quis of order $2$
	if and only if the shifted maps $\phi_1=\varphi_1[1]$ and $\phi_2=
	\varphi_2[1]$ satisfy for all $y,y'\in \mathfrak{H}$
	\begin{eqnarray}
	\lefteqn{0=b\circ \phi_1,~~~\pi\circ\phi_1~\mathrm{invertible},~\mathrm{and}} \nonumber \\
	& &~~~~
	0=b\big(\phi_2(y,y')\big)
	  +\big[\phi_1(y),\phi_1(y')\big]_G
	       -\phi_1\big([y,y']_H\big).
	       \label{EqCompShiftedLInfinityOrder2}
	\end{eqnarray} 
   \end{enumerate} 
\end{lemma}
\begin{proof}
 \textbf{1.} Choose a vector space complement $\mathcal{H}$ of
 $B\mathfrak{G}$ in $Z\mathfrak{G}$, then the restriction of
 $\pi$ to $\mathcal{H}$ is clearly invertible. Let $\phi_1$ be
 the inverse of this map followed by the inclusion of cocycles.
 Then $\pi\circ \phi_1=\mathrm{id}_\mathfrak{H}$, and the
 shift $\varphi_1=\phi_1[1]$
 gives the desired section. Moreover, since $\pi$ is a morphism
 of graded Lie algebras, it follows that for all $y_1,y_2\in
 \mathfrak{H}$ the difference
 $[\phi_1(y_1),\phi_1(y_2)]_G-\phi_1\big([y_1,y_2]_H\big)$
 (which obviously is in $Z\mathfrak{G}$) lies in the kernel of $\pi$ and therefore is
 in $B\mathfrak{G}$, thus proving the existence of a linear map $\varphi_2$ of degree $-1$
 such that $P_2(\varphi)=0$.\\
 \textbf{2.} Since the component $\psi_1$ is a quis, it follows that
 $\pi[1]\circ \psi_1$ is a linear isomorphism of $\mathfrak{H}[1]$
 which intertwines $d_2$ which can be seen by applying $\pi[1]$ to the
 right hand side of the equation $0=P_2(\psi)$. Hence $\alpha=\alpha_1$
 can be defined as the inverse of this map. Then $\tilde{\Psi}
 =\Psi\circ e^{*\alpha_1}$ is a $L_\infty$-quis such that 
 $\tilde{\psi}_1$ is a section. It follows that the difference
 $\tilde{\psi}_1-\varphi_1$ is a linear map from
 $\mathfrak{H}[1]$ into the  space of all coboundaries $B\mathfrak{G}[1]$. Upon choosing a graded vector space complement $W'$
 of $Z\mathfrak{G}[1]$ in $\mathfrak{G}[1]$ we can define a linear
 map $\chi_1:\mathfrak{G}[1]\to\mathfrak{G}[1]$ of degree $-1$ vanishing
 on the coboundaries and on $W'$ and having all its values in $W'$
 such that $\tilde{\psi}_1-\varphi_1= b[1]\circ \chi_1\circ\varphi_1$.
 Thanks to the definition of $\chi_1$ it follows that
 $\chi_1\circ \chi_1=0$, $\chi_1\circ b[1]=0$, and 
 $b[1]\circ \chi_1\circ\tilde{\psi}_1=b[1]\circ \chi_1\circ\varphi_1$.
 Moreover the graded linear map $T$ of degree $0$ defined by
 $T= -[\overline{D},\overline{\chi_1}]$ is a graded coderivation of degree $0$ of
 the graded connected coalgebra $\Sym (\mathfrak{G}[1])$ which can
 be computed (in a straight-forward way) to be
 locally nilpotent in the sense that for each element 
 $c\in\Sym (\mathfrak{G}[1])$ there is a positive integer $N$ such that
 $T^{\circ N}(c)=0$. Therefore the composition exponential 
 $\mathcal{B}=e^{\circ T}$ is a well-defined morphism of connected coalgebras $\Sym (\mathfrak{G}[1])\to\Sym (\mathfrak{G}[1])$ 
 commuting with $\overline{D}$ since obviously $[\overline{D},T]=0$.
 Hence $\Psi'=\mathcal{B}\circ \tilde{\Psi}$ is a $L_\infty$-quis,
 (and hence of the form $e^{*\psi'}$),
 and a straight-forward computation gives $\psi'_1=\varphi_1$.
 Finally equation $P_2(\psi')=0=P_2(\varphi)$ implies that
 $0=b[1]\circ (\psi'_2-\varphi_2)$ proving the second statement of the Lemma.\\
 \textbf{3.} This follows directly from the shifted versions
 of eqs
 (\ref{EqCompP1}) and (\ref{EqCompP2}) upon
  using the rules
 (\ref{EqDefShiftJPlusJPrimeAndCirc}) and
 (\ref{EqDefShiftTensorProduct}).
\end{proof}
In order to prepare the grounds for the characteristic $3$-class,
we recall one variant of \emph{graded Chevalley Eilenberg cohomology
	of a graded Lie algebra 
	$\big(\mathfrak{a},[~,~]_\mathfrak{a}\big)$}: the shifted
	Lie bracket $d_2=[~,~]_\mathfrak{a}[1]$ is an element of
	$\Hom \big(\Sym  (\mathfrak{a}[1]),\mathfrak{a}[1]\big)$ of degree
	$1$ satisfying $[d_2,d_2]_{NR}=0$. Hence the linear map of
	degree $1$ from 
	$\Hom \big(\Sym  (\mathfrak{a}[1]),\mathfrak{a}[1]\big)$ 
	to itself, sending
	$g$ to $[d_2,g]_{NR}$ is a codifferential. Applying the shift 
	$[-1]$ yields a codifferential $\delta_\mathfrak{a}$ on the space
	$\Hom (\Lambda \mathfrak{a},\mathfrak{a})$. 
	The following formula can be computed
	in a straight-forward manner for any homogeneous $\phi_k$ in $\Hom (\Lambda^k \mathfrak{a},\mathfrak{a})$
	and homogeneous elements $y_1,\ldots,y_{k+1}\in\mathfrak{a}$ by computing the shift
	using the formulas (\ref{EqDefShiftJPlusJPrimeAndCirc}) 
	and (\ref{EqDefShiftTensorProduct}):
	\begin{eqnarray}
	  \lefteqn{
	  (\delta_\mathfrak{a}\phi_k)(y_1,\ldots,y_{k+1})} 
	\nonumber \\
	&=&
	\sum_{i=1}^{k+1}(-1)^{i-1}(-1)^{|\phi_k||y_i|}
	              (-1)^{|y_i|(|y_1|+\cdots+|y_{i-1}|)}
	              \nonumber \\
	& &~~~~~~~~~
       \big[y_i,\phi_k(y_1,\ldots,y_{i-1},y_{i+1},\ldots,
	              y_{k+1})\big]_\mathfrak{a}\nonumber \\
	& & + \sum_{1\leq i<j\leq k+1}(-1)^{i+j}
	               (-1)^{(|y_i|+|y_j|)
	               	  (|y_1|+\cdots+|y_{i-1}|)}
	               	(-1)^{|y_j|(|y_{i+1}|+\cdots+|y_{j-1}|)}
	               	\nonumber \\
	& & ~~~~~~~~~
	    \phi_k\big([y_i,y_j]_\mathfrak{a},
	    y_1,\ldots,y_{i-1},y_{i+1},\ldots,y_{j-1},y_{j+1},
	    \ldots,y_{k+1}\big).
	    \label{EqCompGradedChevalleyEilenberg}
	\end{eqnarray}
	where we have left out an inessential global factor of
$-(-1)^{\frac{k(k-1)}{2}}$ appearing in the computation.
In case all arguments $y_1,\ldots,y_{k+1}$ are of even degree	
the usual formula --going back to the exterior derivative of
$k$-forms-- is easily recognised.
Note that $\delta_\mathfrak{a}$ increases the number of arguments
by one, but is of degree zero since the graded Lie bracket is.

We state the following well-known fact in order to have 
a concrete formula which can be computed in examples:
\begin{prop}\label{PW3Z3C3}
	With the above-mentioned definitions let
	$\varphi:\Sym (\mathfrak{H}[1])\to \mathfrak{G}[1]$
	be a linear map of degree $0$ vanishing on $\un$, and suppose that $e^{*\varphi}$ defines a $L_\infty$-quis of
	order $2$ where $\varphi_1$ is a section. Let
	$\phi=\varphi[-1]:\Lambda \mathfrak{H}\to \mathfrak{G}$
	be the shift of $\varphi$. 
	Then the following holds:
	\begin{enumerate}
		\item The linear map $w_3=w_3(\phi):\Lambda^3 \mathfrak{H}\to
		  \mathfrak{G}$ of degree $-1$ defined on homogeneous
		  elements $y_1,y_2,y_3\in\mathfrak{H}$ by
		  \begin{eqnarray}
		  \lefteqn{w_3(\phi)(y_1,y_2,y_3)  = } \nonumber \\
		 & & (-1)^{|y_1|}\big[\phi_1(y_1),\phi_2(y_2,y_3)\big]_G
		  -(-1)^{|y_2|} (-1)^{|y_2||y_1|}
		  \big[\phi_1(y_2),\phi_2(y_1,y_3)\big]_G
		  \nonumber \\
		& & ~~~~~~~~~~~~~+ 
		   (-1)^{|y_3|}(-1)^{|y_3|(|y_1|+|y_2|)}
		  \big[\phi_1(y_3),\phi_2(y_1,y_2)\big]_G
		  \nonumber \\
		& & -\phi_2\big([y_1,y_2]_H, y_3\big)
		     + (-1)^{|y_3||y_2|}
		     \phi_2\big([y_1,y_3]_H, y_2\big)
		     \nonumber \\
		& & 
		~~~~~~~~~~~~~~~~-(-1)^{(|y_2|+|y_3|)|y_1|}
		   \phi_2\big([y_2,y_3]_H, y_1\big)
		   \label{EqDefWThree}
		  \end{eqnarray}
		satisfies $b\circ w_3=0$.
		\item The linear map 
		 \begin{equation}\label{EqDefZThree}
		 z_3=z_3(\phi)=\pi\circ w_3
		 \end{equation} 
		 from
		$\Lambda^3 \mathfrak{H}$ to $\mathfrak{H}$
		is a well-defined graded Chevalley-Eilenberg $3$-cocycle of degree $-1$, 
		\ie $\delta_\mathfrak{H}z_3=0$. 
		\item The graded Chevalley-Eilenberg $3$-class
		 $c_3=c_3\big(\mathfrak{G},b,[~,~]_G\big)$ of $z_3(\phi)$ does not depend on the chosen $\phi_1,\phi_2$ satisfying eqn
		 (\ref{EqCompShiftedLInfinityOrder2}).
		\item There is a $L_\infty$-quis of order 3 between
		 $\mathfrak{G}$ and its cohomology $\mathfrak{H}$
		 if and only if $c_3=0$.
	\end{enumerate}
\end{prop}
\begin{proof}
	Note that the following equation is trivially satisfied
	for any linear map of degree zero $\varphi:\Sym (\mathfrak{H}[1])\to \mathfrak{G}[1]$ vanishing on the unit:
	\begin{equation}\label{EqCompDPPlusPdEqualsZero}
	     \overline{D}\circ \overline{P}(\varphi)
	      + \overline{P}(\varphi)\circ \overline{d_2}=0.
	\end{equation}
	\textbf{1.} Projecting the preceding identity
	(\ref{EqCompDPPlusPdEqualsZero}) to $\mathfrak{G}[1]$ and evaluating it
	on ${\Sym}^3(\mathfrak{H}[1])$ we derive the first statement $b\circ w_3=0$ upon using $P_1(\varphi)=0$ and $P_2(\varphi)=0$
	and applying the shift $[-1]$ to the term $P_3$ (see eqn (\ref{EqCompP3})).\\
	\textbf{2.} By the preceding part, the values of $w_3$
	are cocycles (w.r.t.~$b$) whence the application of
	$\pi$ is legal, and $z_3$ is well-defined. Evaluating
	eqn (\ref{EqCompDPPlusPdEqualsZero}) on  ${\Sym}^4(\mathfrak{H}[1])$, projecting onto 
	$\mathfrak{G}[1]$, applying $\pi[1]$, and using the shift
	$[-1]$ we get the graded $3$-cocycle equation $\delta_\mathfrak{H}z_3=0$, compare eqn
	(\ref{EqCompGradedChevalleyEilenberg}) for $k=3$.\\
	\textbf{3.} Choose a different linear map of degree
	zero $\psi$ from $\Sym (\mathfrak{H}[1])$ to $\mathfrak{G}[1]$ vanishing on the unit with $\psi_1$
	a section such that $P_1(\psi)=0$ and
	$P_2(\psi)=0$. According to the second statement of
	Lemma \ref{LP1P2Vanishing} there is a locally nilpotent
	coderivation $T$ of degree $0$ commuting with $\overline{D}$ such that for the regauged
	map $\psi'=\mathrm{pr}_{\mathfrak{G}[1]}\circ
	e^{\circ T}\circ e^{*\varphi}$ we have $\psi'_1=\varphi_1$.
	Moreover, using eqn (\ref{EqCompGaugingP}) and using
	$P_1(\psi)=0$ and $P_2(\psi)=0$ we get
	\[
	   P_3(\psi')=\mathrm{pr}_{\mathfrak{G}[1]}
	      \circ e^{\circ T}\circ 
	      \overline{P}(\psi)|_{{\Sym}^3(\mathfrak{H}[1])}
	      = P_3(\psi) + b[1]\circ \mathrm{something},
	\]
	showing that $\pi[1]\circ P_3(\psi)=\pi[1]\circ P_3(\psi')$. According to Lemma \ref{LP1P2Vanishing}
	the map $\psi'_2-\varphi_2=\chi_2$ takes its values in
	the $b[1]$-cocycles. It follows that the projection $\pi$ applied to the difference
	$w_3(\psi')-w_3(\phi)$ gives the graded $3$-coboundary
	$\delta_\mathfrak{H}(\chi_2[-1])$ because the $b$-cocycles
	form a graded Lie subalgebra, and $\pi$ is a morphism
	of graded Lie algebras. Hence modulo graded $3$-coboundaries
	the expression $c_3$ is independent on $\phi_1,\phi_2$.\\
	\textbf{4.} If there is a $L_\infty$-quis of order three
	there is a regauged one, $e^{*\varphi}$, where $\varphi_1$ is a section
	according to the proof of Lemma \ref{LP1P2Vanishing}. But
	then $P_3(\varphi)=0$, hence $z_3(\phi)=0$, and the class
	$c_3$ vanishes.\\
	Conversely, suppose the class $c_3$ vanishes. Choose any
	$L_\infty$-quis $e^{*\varphi}$ of order $2$ with $\varphi_1$
	a section, which exists by Lemma \ref{LP1P2Vanishing}.
	We can then add to $\phi_2$ a linear map $\chi_2[-1]$ taking values in the $b$-cocycles such 
	that for $\phi_2'=\phi_2+\chi_2[-1]$ the projection
	$z_3(\phi')=0$. This means that $w_3(\phi')$ is a map
	into the $b$-coboundaries, whence after a shift we get
	a linear map $\varphi'_3$ such  that $P_3(\varphi')=0$
	whence $e^{*\varphi}$ is a $L_\infty$-quis of order $3$.
\end{proof}

We shall call the above-mentioned graded cohomology $3$-class
$c_3\big(\mathfrak{G},b,[~,~]_G\big)$ the 
\emph{characteristic $3$-class of the differential graded
	Lie algebra $\big(\mathfrak{G},b,[~,~]_G\big)$}.
	It obviously is the first obstruction to $L_\infty$-formality and can be computed with any
	$\phi_1$ and $\phi_2$ satisfying eqs (\ref{EqDefSection})
	and (\ref{EqCompShiftedLInfinityOrder2}).

\section{$L_\infty$-Perturbation Lemma} \label{Sec:Perturbation-lemma}

We refer to Appendix \ref{App: The Perturbation Lemma} for
homotopy contractions and the usual Perturbation Lemma.

In case the corresponding complexes carry additional algebraic
structures it is interesting to see whether the maps
in the perturbation lemma can be modified such that these
structures are preserved. This has been done by an inductive
procedure in the
$A_\infty$ and $L_\infty$ cases, and on other operads, see e.g. \cite{Hueb10,Hueb11}, \cite[Prop A.3]{BGHHW05}, \cite{M10,DSV16}.
 We would like to present the observation that in the 
 $L_\infty$-case the usual geometric series will already
 give maps preserving the graded coalgebra structures.\\
To approach the $L_\infty$-perturbation Lemma we consider
an arbitrary homotopy contraction \eqref{DiagramHomotopyContraction}. One can pass to the
graded coderivations $\overline{b_U[1]}$ of $\Sym (U[1])$ and
$\overline{b_V[1]}$ of $\Sym (V[1])$, respectively, and 
--upon writing $\varphi_1$ for $i[1]$ and $\psi_1$ for
$p[1]$-- consider
the morphisms of graded coalgebras $e^{*\varphi_1}:
\Sym (U[1])\to \Sym(V[1])$ and 
$e^{*\psi_1}:\Sym(V[1])\to \Sym (U[1])$, respectively. By applying the corresponding projections
$\pr_{U[1]}$ and $\pr_{V[1]}$ it can be seen that $\overline{b_U[1]}$
and $\overline{b_V[1]}$ are differentials, and 
$e^{*\varphi_1}$ and $e^{*\psi_1}$ are chain maps satisfying
$e^{*\psi_1}\circ e^{*\varphi_1}=id_{\Sym (U[1])}$. 
In order to extend the chain homotopy $h$ from $V$
to $\Sym (V[1])$, the simple choice $\overline{h[1]}$ 
(as a graded coderivation and derivation) will not be
enough. Recall that $P=[h,b_V]$ is an idempotent $\korps$-linear map $V\to V$. Let $V_U$ be the kernel, and $V_\mathrm{acyc}$
be the image of $P$. Clearly $V=V_U\oplus V_\mathrm{acyc}$,
and $\Sym (V[1])\cong \Sym (V_U[1])\otimes\Sym (V_\mathrm{acyc}[1])$
as graded bialgebras. Define the $\korps$-linear map $\beta$
of degree $0$ from $\Sym (V[1])$ to $\Sym (V[1])$
for all $y_1,\ldots,y_k\in V_U[1]$ and $w_1,\ldots,w_l\in
V_\mathrm{acyc}[1]$ where $k,l\in\mathbb{N}$:

\begin{align}\label{EqDefBeta}
\beta(y_1 \bullet \dotsb \bullet y_k \bullet w_1 \bullet \dotsb \bullet w_l) & {}=
\begin{cases}
\frac{1}{l}(y_1 \bullet \dotsb \bullet y_k 
\bullet w_1 \bullet \dotsb \bullet w_l) & \text{if $l \neq 0$}, \\
0 & \text{if $l=0$},
\end{cases}
\end{align}
and set
\begin{equation}\label{EqDefHomotopyEta}
\eta =\overline{h[1]}\circ \beta =\beta \circ \overline{h[1]}.
\end{equation}
It is then easy to see that

 \begin{equation}\label{DiagramHomotopyContractionSymmetric}
  \begin{tikzpicture}[baseline=(current bounding box.183)]
  \matrix (m) [matrix of math nodes,nodes in empty cells,column sep=1em,text height=1.5ex,text depth=0.25ex]
  {\big(\Sym (U[1]),\overline{b_U[1]}\big) 
  	& & \big(\Sym (V[1]),\overline{b_V[1]}\big)  & \\};
  \path[right hook->]
  ([yshift=-5pt]m-1-1.north east) edge node [above] 
      {$e^{*\varphi_1}$} ([yshift=-5pt]m-1-3.north west);
  \path[->>]
  ([yshift=5pt]m-1-3.south west) edge node [below] 
      {$e^{*\psi_1}$} ([yshift=5pt]m-1-1.south east);
  \draw[->] (m-1-4.north) arc (120:-120:2ex);
  \draw (m-1-4) ++(2em,0em) node {$\eta$};
  \end{tikzpicture}
\end{equation}  
is a homotopy contraction. Now the following Theorem is
quite useful since it allows to transfer $L_\infty$-structures
via homotopy contraction from $V$ to $U$.

\begin{theorem} \label{Thm:PertBordemann}
	Let $(U,b_U)$ and $(V,b_V)$ be two chain complexes. Suppose that there is a homotopy contraction \eqref{DiagramHomotopyContraction}. In the corresponding
	shifted symmetric algebra version \eqref{DiagramHomotopyContractionSymmetric}, suppose
	that there is a $\korps$-linear map $D'_V=\sum_{k\geqslant 2}D'_k:
	\Sym (V[1])\to V[1]$ of degree $1$ such that 
	$D=b_V[1]+D'_V$ defines an
	$L_\infty$-structure, \ie the coderivation
	$\overline{b_V[1]+D'_V}$ is a differential whence
	the coderivation
	$\delta_{\Sym (V[1])}=\overline{D'_V}$ is a perturbation of
	$\overline{b_V[1]}$.
	
	Then the $\korps$-linear maps 
	$\widetilde{e^{*\varphi_1}}$,
	$\widetilde{e^{*\psi_1}}$,
	$\delta_{\Sym (U[1])}$, and $\widetilde{\eta}$
	of the Perturbation Lemma \ref{Lem:Perturbation} --
	which define a homotopy contraction between complexes,
	\begin{equation*}
	\begin{tikzpicture}[baseline=(current bounding box.center)]
	\matrix (m) [matrix of math nodes,nodes in empty cells,column sep=1em,text height=1.5ex,text depth=0.25ex]
	{\big(\Sym(U[1]),\overline{b_U[1]}
		+\delta_{\Sym (U[1])} \big) & & \big( \Sym (V[1]),\overline{b_V[1]}+\overline{D'_V}\big) & \\};
	\path[right hook->]
	([yshift=-5pt]m-1-1.north east) edge node [above] {$\widetilde{e^{*\varphi_1}}$} ([yshift=-5pt]m-1-3.north west);
	\path[->>]
	([yshift=5pt]m-1-3.south west) edge node [below] {$\widetilde{e^{*\psi_1}}$} ([yshift=5pt]m-1-1.south east);
	\draw[->] (m-1-4.north) arc (120:-120:2ex);
	\draw (m-1-4) ++(2em,0em) node {$\widetilde{\eta}$};
	\end{tikzpicture}
	\end{equation*}
	will \textbf{automatically} preserve the structure of
	graded connected coalgebras, \ie $\widetilde{e^{*\varphi_1}}$ and 
	$\widetilde{e^{*\psi_1}}$ are morphism of graded differential connected coalgebras, and 
	$\delta_{\Sym (U[1])}$ will be a graded coderivation
	of degree $1$.\\
	More explicitly, defining the $\korps$-linear maps of 
	degree $0$, $\varphi=\varphi_1+\sum_{k \geqslant 2}\varphi_k:  
	\Sym (U[1])\to V[1]$ and
	$\psi=\psi_1+\sum_{k \geqslant 2}\psi_k:
	\Sym (V[1])\to U[1]$ and the $\korps$-linear map
	$d'_U=\sum_{k \geqslant 2}d'_k:
	\Sym (U[1])\to U[1]$ of degree $1$ by
	\begin{subequations}
	\begin{gather}
	 \varphi =\pr_{V[1]}\circ 
	 (id_{\Sym(V[1])} + \eta \circ\overline{D'_V})^{-1}
	 \circ \varphi_1,\label{EqDefLInfinityPertLemmaInjection}\\
	 \psi =\psi_1\circ \pr_{V[1]}\circ 
	 (id_{\Sym(V[1])} + \overline{D'_V}\circ \eta)^{-1},
	    \label{EqDefLInfinityPertLemmaProjection} \\
	 d'_U=\psi_1\circ \pr_{V[1]}\circ (id_{\Sym(V[1])} + \overline{D'_V} \circ \eta)^{-1} \circ \overline{D'_V} \circ e^{*\varphi_1}\label{EqDefInducedDifferentialHPTLInf}
	\end{gather}
	\end{subequations}
	we get
	\begin{subequations}
	\begin{gather}
	e^{*\varphi}
	= \widetilde{e^{*\varphi_1}} =
	\big(id_{\Sym(V[1])} + \eta \circ \overline{D'_V}\big)^{-1}
	\circ e^{*\varphi_1}, \\
	e^{*\psi}
	= \widetilde{e^{*\psi_1}} =
	e^{*\psi_1} \circ 
	\big(id_{\Sym(V[1])} 
	   + \overline{D'_V} \circ \eta\big)^{-1}, \\
	\overline{d'_U}=
		\delta_{\Sym (U[1])}=   e^{*\psi_1} 
	\circ (id_{\Sym(V[1])} + \overline{D'_V} \circ \eta)^{-1} \circ \overline{D'_V} \circ e^{*\varphi_1},
	\intertext{and of course the perturbed
	chain homotopy}
	\widetilde{\eta} = 
	(id_{\Sym(V[1])} +  \eta \circ \overline{D'_V})^{-1} \circ \eta.
	\end{gather}
	\end{subequations}
	This entails in particular that $e^{*\varphi}$ is a $L_\infty$-quasi-isomorphism with quasi-inverse $e^{*\psi}$.	
\end{theorem}
\noindent A proof of this is given in \cite{BE18}.

\noindent The preceding Theorem \ref{Thm:PertBordemann} has the following well-known corollary:
\begin{cor}
	Suppose that in the homotopy contraction (\ref{DiagramHomotopyContraction}) the differential
	$b_U=0$ whence $U$ is isomorphic to the cohomology of
	$(V,b_V)$. Then under the hypothesis of the preceding
	Theorem \ref{Thm:PertBordemann} we have the following:
	\begin{enumerate}
		\item There is a minimal $L_\infty$-structure $(U,d)$ 
		 quasi-isomorphic to $(V,b_V[1]+D'_V)$. \\
		 In case
		 the latter $L_\infty$-structure comes from the
		 structure of a differential graded Lie algebra,
		 $(V,b,[~,~])$, then the term $d_2$ is isomorphic
		 to the shift of the induced Lie bracket $[~,~]_H$ on cohomology, $[~,~]_H[1]$.
    \item Suppose there is another homotopy contraction 
	\begin{equation*}
	\begin{tikzpicture}[baseline=(current bounding box.184)]
	\matrix (m) [matrix of math nodes,nodes in empty cells,column sep=1em,text height=1.5ex,text depth=0.25ex]
	{(U',0) & & (V,b_V) & \\};
	\path[right hook->]
	([yshift=-5pt]m-1-1.north east) edge node [above] {$i'$} ([yshift=-5pt]m-1-3.north west);
	\path[->>]
	([yshift=5pt]m-1-3.south west) edge node [below] {$p'$} ([yshift=5pt]m-1-1.south east);
	\draw[->] (m-1-4.north) arc (120:-120:2ex);
	\draw (m-1-4) ++(2em,0em) node {$h'$};
	\end{tikzpicture}.
	\end{equation*}
	Then, writing $\varphi'_1=i'[1]$, $\psi'_1=p'[1]$ under the hypothesis of the preceding Theorem \ref{Thm:PertBordemann} the two $L_\infty$ structures $(U,d)$ and $(U',d')$ are conjugated, \ie there is an isomorphism of coalgebras 
	$e^{*\chi} : \Sym(U[1]) \to \Sym(U'[1])$ such that
	\begin{equation}
	\overline{d'} = e^{*\chi}\circ \overline{d} \circ \left(e^{*\chi}\right)^{-1}.
	\end{equation}
	\end{enumerate}
\end{cor}

\begin{proof}
	\begin{enumerate}
	\item The first statement is immediate from the
	preceding Theorem \ref{Thm:PertBordemann}.
	For the second, according to equation (\ref{EqDefInducedDifferentialHPTLInf}) we have
	$d'_2=\psi_1\circ D'_2\circ (\varphi_1 \bullet \varphi_1)$,
	and for $D'_2=[~,~][1]$ it is clear that
	this is thus the shifted induced bracket on cohomology 
	with isomorphism given by $\varphi_1$.\\
	\item
	Using restrictions, since $e^{*\psi'}|_{V[1]} = \psi'_1$  induces an isomorphism $V[1] \to U'[1]$ and 
	$e^{*\varphi}|_{U[1]} = \varphi_1$ induces an isomorphism $U[1] \to V[1]$, we have that $(e^{*\psi'} \circ e^{*\varphi})|_{V[1]} = \psi'_1 \circ \varphi_1$ is an isomorphism $U[1] \to U'[1]$, which implies that
	the morphism of graded connected coalgebras 
	$e^{*\chi}\coloneqq e^{*\psi'} \circ e^{*\varphi}$ is invertible using \autoref{Lem:MorphCoInv} above.
	We have
	\(
	e^{*\psi'} \circ e^{*\varphi} \circ \overline{d} = e^{*\psi'} \circ \overline{D} \circ e^{*\varphi} = \overline{d'} \circ e^{*\psi'} \circ e^{*\varphi}
	\)
	proving the Corollary.
	\end{enumerate}
\end{proof}

\noindent We get the following $L_\infty$-analogue of Remark
\ref{RemarksContraction}, \ref{RemSubcomplexLeadsToContraction}
which is quite useful:
\begin{cor}\label{CorLInfinityMapwillBeQuis}
	Let $(U,b_U)$ and $(V,b_V)$ be complexes and suppose there
	is a homotopy contraction 
	(\ref{DiagramHomotopyContraction}). Let furthermore
	$(U,b_U[1]+D_U)$ and $(V,b_V[1]+D_V)$ be 
	$L_\infty$-structures and 
	\[
	\Phi:
	\big(\Sym (U[1]),\overline{b_U[1]+D_U}\big)\to 
	\big(\Sym (V[1]),\overline{b_V[1]+D_V}\big)
	\]
	 an 
	$L_\infty$-map such that (writing $\psi'_1=p[1]$ and 
	$\varphi'_1=i[1]$)
	\[
	     \psi'_1\circ \Phi|_{U[1]}:U[1]\to U[1]
	\]
	is invertible.\\
	Then $\Phi$ is a $L_\infty$-quasi-isomorphism.
	Moreover, if both $L_\infty$-structures come from 
	differential graded Lie algebra structures on $U$ and
	$V$, respectively, then the corresponding graded Lie structures on the cohomologies of $U$ and $V$ with
	respect to $b_U$ and $b_V$, respectively, are isomorphic.
\end{cor}
\begin{proof}
	Write $\Phi|_{U[1]} \eqqcolon \varphi_1 : U[1] \to V[1]$, 
	$A:=\psi'_1\circ \varphi_1$ the invertible $\korps$-linear
	map which clearly is a chain map
	$(U[1],b_U[1])\to(U[1],b_U[1])$. Set $\hat{\psi}_1:=A^{-1}\circ \psi'_1$.
	Then $\hat{\psi}_1$ is a chain map 
	$(V[1],b_V[1])\to(U[1],b_U[1])$ and clearly
	\[
	     \hat{\psi}_1\circ \varphi_1 
	     =A^{-1}\circ \psi'_1\circ \varphi_1 = id_{U[1]}.
	\]
	On the other hand $\varphi_1\circ \hat{\psi}_1$ commutes
	with $b_V[1]$, and thanks to eqn 
	(\ref{contraction_eq2}) we get
	\begin{align*}
	 \varphi_1\circ \hat{\psi}_1 
	 = &\big(\varphi'_1\circ \psi'_1 + \big[h[1],b_V[1]\big]\big)
	     \circ \varphi_1\circ \hat{\psi}_1 \\
	 = & \varphi'_1\circ \psi'_1
	     + \big[h[1]\circ \varphi_1\circ \hat{\psi}_1 ,b_V[1]\big]\\
	 = & id_{V[1]} -\big[h[1]\circ (id_{V[1]}-\varphi_1\circ \hat{\psi}_1) ,b_V[1]\big],
	\end{align*}
	whence the $\korps$-linear maps $\hat{\psi}_1,\varphi_1,
	h'[1]\coloneqq h[1]\circ (id_{V[1]}-\varphi_1\circ \hat{\psi}_1) $
	define a homotopy contraction (\ref{DiagramHomotopyContraction}) for the complexes
	$(U[1],b_U[1])$ and $(V[1],b_V[1])$, the check of the side conditions for $h'[1]$ being straight-forward.
	In particular, $\varphi_1$ induces an isomorphism in 
	cohomology whence it is a $L_\infty$-quasi-isomorphism.\\
	For the second statement let $D_U=D_2=[~,~][1]$
	and $D_V=D'_2=[~,~]'[1]$ where $[~,~]$ and $[~,~]'$
	denote the graded Lie brackets on $U$ and $V$, respectively.
	Pick a quasi-inverse $\Psi:\Sym(V[1])\to \Sym(U[1])$
	which exists, see e.g. \cite[Thm. V1, V2]{AMM02}, then 
	the fact that both $\Phi= e^{*\varphi}$ and 
	$\Psi=e^{*\psi}$ are chain maps read when evaluated on
	two elements and projected to the primitive space:
	for all $y_1,y_2\in U[1]$ and $z_1,z_2\in V[1]$ we get
	with $d_U=b_U[1]$ and $d_V=b_V[1]$
	\begin{gather}
	  \varphi_2\big(
	   d_U(y_1)\bullet y_2+(-1)^{|y_1|}y_1\bullet d_U(y_2)\big)
	      + \varphi_1\big(D_2(y_1\bullet y_2)\big)\nonumber\\
	      = d_V\big(\varphi_2(y_1\bullet y_2)\big)
	      + D'_2\big(\varphi_1(y_1)\bullet\varphi_1(y_2)\big)
	\end{gather}
	and 
	\begin{gather}
	\psi_2\big(
	d_V(z_1)\bullet z_2+(-1)^{|z_1|}z_1\bullet d_U(z_2)\big)
	+ \psi_1\big(D'_2(z_1\bullet z_2)\big)\nonumber\\
	= d_U\big(\psi_2(z_1\bullet z_2)\big)
	+ D_2\big(\psi_1(z_1)\bullet\psi_1(z_2)\big).
	\end{gather}
    Hence if $y_1,y_2,z_1,z_2$ are cocycles, it follows that
    $\varphi_1$ and its inverse $\psi_1$ preserve Lie brackets
    up to coboundaries, and upon projecting onto the corresponding cohomology, we get the desired isomorphism
    of graded Lie brackets.
\end{proof}

For later use, in the case $b_U=0$ we give the shifted formula for the linear map $\phi_2=\varphi_2[-1]$ (of degree $-1$)
from
$\Lambda^2\mathfrak{H}$ to $\mathfrak{G}$ in case of a
differential graded Lie algebra
 $\big(\mathfrak{G},b,[~,~]_G\big)$ 
 (whence $D=b[1]+[~,~]_G[1]$), \ie
 \begin{equation}\label{EqCompPhiTwoPertShifted}
    \phi_2(y_1,y_2)=-h\big([\phi_1(y_1),\phi_1(y_2)]_G\big).
 \end{equation}
 which can be used to compute the characteristic $3$-class
 $c_3$ from $w_3$, see eqn (\ref{EqDefWThree}).

\section{Finite-dimensional Lie algebras} \label{Sec:Finite-dim-Lie-alg}

In this Section we shall consider different examples of finite-dimensional (trivially graded)
Lie algebras $\big(\mathfrak{g}, [~,~]\big)$, and study the formality of the Hochschild complex $\Hochcochains(\UEA\mathfrak{g},\UEA\mathfrak{g})$ of their universal enveloping algebra $\UEA\mathfrak{g}$. 

Recall the well-known \emph{Chevalley-Eilenberg complex} $\big(\ChEcochains(\mathfrak{g},\Sym\mathfrak{g}),
\delta\big)$ of the Lie-algebra $\mathfrak{g}$ taking values
in the symmetric algebra $\Sym \mathfrak{g}$ seen as a
$\mathfrak{g}$-module via the adjoint representation:
Since $\mathfrak{g}$ is finite-dimensional it is canonically
isomorphic to the tensor product $\Sym\mathfrak{g}\otimes \Lambda\mathfrak{g}^*$ and is $\mathbb{Z}$-graded by the
`Grassmann degree', i.e. the form degree of the second factor $\Lambda\mathfrak{g}^*$.
With this grading, it clearly is a graded commutative algebra
by means of the tensor product of the commutative multiplication in $\Sym\mathfrak{g}$ and the usual exterior
multiplication in $\Lambda \mathfrak{g}^*$ which we shall also
denote by $\wedge$.
Considering an element $f\in\Sym \mathfrak{g}$ as a polynomial function on the dual space $\mathfrak{g}^*$ we can consider
$\ChEcochains(\mathfrak{g},\Sym\mathfrak{g})$ 
as the space of all
polynomial poly-vector-fields on $\mathfrak{g}^*$ equipped with
the usual \emph{Schouten bracket} $[~,~]_S$, see e.g.
\cite[eq.(5.1)]{BM08} or \cite{BMP05}; we recall the definition: let $e_1,\ldots,e_n$
be a fixed basis of $\mathfrak{g}$, let $\epsilon^1,\ldots,
\epsilon^n$
denote the dual basis of $\mathfrak{g}^*$, and recall the
\emph{structure constants} of the Lie algebra $\mathfrak{g}$, $c^i_{jk}=\epsilon^i\big([e_j,e_k]\big)\in\korps$. Then each $x\in\mathfrak{g}^*$ can be written as a sum
$x=\sum_{i=1}^nx_i\epsilon^i$. 
For each $\xi\in\mathfrak{g}$ we have the usual interior product
graded derivation 
$\iota_\xi:\Lambda \mathfrak{g}^*\to \Lambda \mathfrak{g}^*$,
and for each $y\in\mathfrak{g}^*$ we have the corresponding
derivation $\iota_y:\Sym\mathfrak{g}\to\Sym\mathfrak{g}$.
For a dual basis vector $e^i$ we shall sometimes write the more suggestive
way $\iota_{\epsilon^i}(f)=\partial^i f$ for each $f\in \Sym\mathfrak{g}$. We extend these derivations to the tensor product $\ChEcochains(\mathfrak{g},\Sym\mathfrak{g})$ in the obvious way and write $\wedge$ for the combined multiplication.
With these conventions the Schouten bracket  
of two elements $F,G\in   \ChEcochains(\mathfrak{g},\Sym\mathfrak{g})[1]$ reads
\begin{equation}\label{EqDefSchoutenBracketFinDim}
  [F,G]_s =\sum_{i=1}^n \iota_{e_i}(F)\wedge \partial^iG
      -(-1)^{(|F|-1)(|G|-1)}
         \sum_{i=1}^n \iota_{e_i}(G)\wedge\partial^iF
\end{equation}
where the degree $|F|$ is the original unshifted `Grassmann degree' to which we have sticked for computational reasons.
Recall that $\big(\ChEcochains(\mathfrak{g},\Sym\mathfrak{g}),\wedge,
[~,~]_s\big)$ is a \emph{Gerstenhaber algebra}, i.e.~there is
a graded Leibniz rule
\begin{equation}
[F,G\wedge H]_s =[F,G]_s\wedge H 
       + (-1)^{(|F|-1)|G|}G\wedge [F,H]_s
\end{equation}
Denoting by 
\begin{equation}\label{EqDefLinearPoissonStructure}
 \pi=[~,~]=
   \frac{1}{2}\sum_{i,j,k}c^i_{jk}e_i\otimes 
   (\epsilon^j\wedge \epsilon^k)
\end{equation}
 the so-called
\emph{linear Poisson structure} of $\mathfrak{g}^*$ we of course
have $[\pi,\pi]_s=0$, and we can use
the coboundary operator $\delta=\delta_{\mathfrak{g}}=[\pi,~]_s$ for the (shifted) Chevalley
Eilenberg coboundary operator which differs from the historical
definition by an unessential global sign.
It thus follows that $\big(\ChEcochains(\mathfrak{g},\Sym\mathfrak{g})[1],
\delta_{\mathfrak{g}}, [~,~]_s \big)$ is a differential graded Lie algebra.
The following Theorem shows that in order to check formality of
the Hochschild complex it suffices to check it for the `easier'
Chevalley-Eilenberg complex. Since we shall need Kontsevich's
formality theorem we shall assume for the rest of this Section that the field $\korps$
is equal to $\mathbb{C}$:

\begin{theorem} \label{TheoFormalityinStages}
	Let $(\mathfrak{g},[~,~])$ be a finite-dimensional complex Lie-algebra. 
	\begin{enumerate}
	\item There is a $L_\infty$-quasi-isomorphism between 
	the differential graded Lie algebra $\big(\ChEcochains(\mathfrak{g},\Sym\mathfrak{g})[1],
	\delta_{\mathfrak{g}}, [~,~]_s \big)$ and the differential graded Lie algebra $\big(\Hochcochains(\UEA\mathfrak{g},\UEA\mathfrak{g})[1],
	b,  [~,~]_G\big)$.\\
	In particular, this induces an isomorphism of graded Lie algebras of their cohomologies (with respect to $\delta_{\mathfrak{g}}$ and $b$, respectively).\footnote{This answers a question asked by F.~Wagemann for the case of finite-dimensional Lie algebras.}
	\item The $L_\infty$-formality of $\big(\ChEcochains(\mathfrak{g},\Sym\mathfrak{g})[1],
		\delta_{\mathfrak{g}}, [~,~]_s \big)$ is equivalent to
		the $L_\infty$-formality of 
	$\big(\Hochcochains(\UEA\mathfrak{g},\UEA\mathfrak{g})[1],
		b,  [~,~]_G \big)$.	
	\end{enumerate}
\end{theorem}

\begin{proof}
	In \cite[Theorem 6.2]{BM08}, a morphism of differential graded
	coalgebras $e^{*\varphi'}$
	from $\Sym \big(\ChEcochains(\mathfrak{g},\Sym\mathfrak{g})[2]\big)$
	to $\Sym\big(\Hochcochains(\UEA\mathfrak{g},\UEA\mathfrak{g})[2]\big)$
	had been constructed by twisting the well-known Kontsevich formality quasi-isomorphism $e^{*\varphi}$
	from $\Sym \big(\ChEcochains(\mathfrak{g},\Sym\mathfrak{g})[2]\big)$
	to $\Sym\big(\Hochcochains(\Sym\mathfrak{g},
	\Sym\mathfrak{g})[2]\big)$
	by the formal exponential $e^{\bullet \hbar\pi}$ of the linear Poisson structure $\pi$ defined
	by the Lie bracket of $\mathfrak{g}$ and observing that
	the resulting map converges for $\hbar=1$ on polynomials. 
	Here the fact that the universal enveloping algebra of 
	$\mathfrak{g}$ can be seen as a converging Kontsevich deformation 
	of the symmetric algebra (sketched in 
	\cite[Secs.~8.3.1,~8.3.2]{K03}) has also been used, see \cite{BMP05} for more details. Now $e^{*\varphi'}$ is even a $L_\infty$-quasi-isomorphism which has just been stated without proof in \cite[Theorem 6.2]{BM08}. We shall indicate the proof of it
	which is relatively straight-forward. Recall the standard
	quasi-isomorphism of the Chevalley-Eilenberg complex of
	$\mathfrak{g}$ with values in $\Sym \mathfrak{g}$ with the
	Hochschild cohomology complex $\Hochcochains(\UEA\mathfrak{g},\UEA\mathfrak{g})$ of its enveloping algebra which we can write in the following way as a contraction of complexes:
	\begin{equation*}
	\begin{tikzpicture}[baseline=(current bounding box.184)]
	\matrix (m) [matrix of math nodes,nodes in empty cells,column sep=1em,text height=1.5ex,text depth=0.25ex]
	{\big(\ChEcochains(\mathfrak{g},\Sym\mathfrak{g}),
		\delta_{\mathfrak{g}}\big)
		& &  
		\big(\Hochcochains(\UEA\mathfrak{g},\UEA\mathfrak{g}),
		b\big)
		& \\};
	\path[right hook->]
	([yshift=-5pt]m-1-1.north east) edge node [above] {$\phi_{CE}$} ([yshift=-5pt]m-1-3.north west);
	\path[->>]
	([yshift=5pt]m-1-3.south west) edge node [below] {$\psi_{CE}$} ([yshift=5pt]m-1-1.south east);
	\draw[->] (m-1-4.north) arc (120:-120:2ex);
	\draw (m-1-4) ++(2em,0em) node {$h$};
	\end{tikzpicture}.
	\end{equation*}
	Here $\psi_{CE}$ is given by $\psi_{CE}(F)=\omega^{-1}\circ F\circ \mathcal{F}$
	where the map $\mathcal{F}:\Lambda^{\bullet} \mathfrak{g}\to 
	\UEA\mathfrak{g}^{\otimes }$ is
	already given in \cite[p.280]{CE56} 
	and consists of evaluation of a Hochschild cochain on
	the antisymmetrization of its arguments restricted to
	$\mathfrak{g}\subset \UEA\mathfrak{g}$, and $\omega:\Sym\mathfrak{g}
	\to \UEA\mathfrak{g}$ is the canonical symmetrization map
	which is an isomorphism of  $\mathfrak{g}$-modules, see e.g.~\cite[p.78]{Dix77}.
	$\phi_{CE}$ is a quasi-inverse of $\psi_{CE}$ and 
	$h$ is a chain homotopy which are much harder to describe explicitly: it can be done in terms of Eulerian idempotents and iterated integrals, see the PhD-thesis of S.~Rivi\`{e}re \cite{Riv12}. By Corollary \ref{CorLInfinityMapwillBeQuis} we just have to check
	whether $\psi_{CE}[2]\circ \varphi'_1$ is an invertible
	map: according to eqn
	(6.1) of Thm 6.1 of \cite{BM08} $\varphi'_1$ takes the
	following form: for any polyvector-field $F$ of rank
	$m$ in $\ChEcochains(\mathfrak{g},\Sym\mathfrak{g})$
	and $m$ polynomials $f_1,\ldots,f_m$ in $\Sym\mathfrak{g}$
	(seen as polynomial functions on $\mathfrak{g}^*$)
	we have (the image of the Kontsevich formality map $\varphi$
	are poly-differential operators)
	\[
	 \varphi_1'(F)(f_1,\ldots,f_m)
	 =\varphi_1(F)(f_1,\ldots,f_m)
	   +\sum_{r=1}^\infty\frac{1}{r!}
	   \varphi_{r+1}(\pi^{\bullet r}\bullet F)(f_1,\ldots,f_m).
	\]
	Let $x_1,\ldots,x_m\in\mathfrak{g}$ seen as linear functions
	on the dual space $\mathfrak{g}^*$. Then 
	\begin{align}\label{EqKontsevichTwistedOneCircleCaEi}
	 &(\psi_{CE}[2]\circ \varphi'_1)(F)(x_1,\ldots,x_m)
	  = \notag \\
	  & \mathrm{Alt}\left(\varphi_1(F)(x_1,\ldots,x_m)
	  +\sum_{r=1}^\infty\frac{1}{r!}
	  \varphi_{r+1}(\pi^{\bullet r}\bullet F)(x_1,\ldots,x_m)\right)
	 \end{align}
	 where $\mathrm{Alt}$ denotes antisymmetrization in the
	 arguments $x_1,\ldots,x_m$. The first term on the
	 right hand side of eqn (\ref{EqKontsevichTwistedOneCircleCaEi}) is up to a nonzero
	 constant factor equal to $\xi(x_1,\ldots,x_m)$. In the second term on the right hand side of eqn (\ref{EqKontsevichTwistedOneCircleCaEi}) involving the
	 sum $\sum_{r=0}^\infty$
	 we check the polynomial degree of the corresponding
	 polyvector-field: due to Kontsevich's universal formula \cite{K03} it follows that 
	 $\varphi_{r+1}(\pi^{\bullet r}\bullet F)$ is a polydifferential operator acting on $m$ functions: it is a finite sum parametrised by certain graphs where in each term 
	 $2r+m$ partial derivatives are distributed over the
	 $r$ linear Poisson structures $\pi$, the polynomial
	 polyvector-field $F$ (where $\delta$ denotes the maximal
	 polynomial degree of its coefficients), and the $m$ functions. In eqn  (\ref{EqKontsevichTwistedOneCircleCaEi}) we need to consider only $m$ linear functions $x_1,\ldots,x_m$, hence $2r+m$ derivatives meet a polynomial of degree $r+\delta+m$,
	 whence the resulting polynomial degree is $\delta-r$:
	 it follows that the above sum in the second part on the right hand side of
	 (\ref{EqKontsevichTwistedOneCircleCaEi}) has at most
	 $\delta$ terms, and the polynomial degree of the resulting
	 polyvector-field is stricly lower than $\delta$.
	 By a simple filtration argument in the polynomial degree
	 it follows that
	 $\psi_{CE}[2]\circ \varphi'_1$ is equal to an invertible
	 map plus lower order terms and is thus invertible proving
	 the first part of the Theorem. \\
	 The second statement is immediate.
\end{proof}

In the following subsections we check formality for
the Chevalley-Eilenberg complex of certain finite-dimensional
Lie algebras, hence we look at the following differential graded
Lie algebras
\begin{equation}
  \big(\mathfrak{G},b,[~,~]_G\big)
   =\big(\ChEcochains (\mathfrak{g},\Sym \mathfrak{g})[1],
          \delta_\mathfrak{g}, [~,~]_s\big)
          ~~~\mathrm{and}~~
     \big(\mathfrak{H},[~,~]_H\big)
     =
     \big(\ChE (\mathfrak{g},\Sym \mathfrak{g})[1],[~,~]_H\big).
\end{equation}
We shall denote the $\delta_\mathfrak{g}$-cohomology classes
of a cocycle $F$ in $\ChEcochains (\mathfrak{g},\Sym \mathfrak{g})$ by $[F]$.

\subsection{\texorpdfstring{Abelian Lie algebras}
	        {Abelian Lie algebras}}

In case the Lie algebra $\mathfrak{g}$ is abelian, $\UEA\mathfrak{g} = \Sym{\mathfrak{g}}$, and then there is nothing to prove since the Chevalley-Eilenberg differential is zero, whence $\ChEcochains(\mathfrak{g},\Sym \mathfrak{g})\cong
\ChE(\mathfrak{g},\Sym \mathfrak{g})$,
and formality of $\ChEcochains(\mathfrak{g},\Sym \mathfrak{g})$ is the content of the Kontsevich formality theorem.

\subsection{\texorpdfstring{Lie algebra of the affine group of $\korps^m$}{Lie algebra of the affine group of {K\textasciicircum m}}}

\begin{theorem}[\protect{\cite[Theorem 6.3]{BM08}}]
	Let $\mathfrak{g}$ be the affine Lie algebra \ie the semidirect sum
	\begin{equation*}
	\mathfrak{gl}(m,\korps) \oplus \korps^m.
	\end{equation*}
	Then the differential graded Lie algebra $\big(\ChEcochains(\mathfrak{g},\Sym \mathfrak{g})[1], \delta_\mathfrak{g},[~,~]_s\big)$ is formal.
\end{theorem}

Here the cohomology is represented by certain `constant poly-vector fields', i.e. by elements of $\Lambda \mathfrak{g}^*$ whose
Schouten brackets all vanish, hence the cohomology injects as a graded Lie sub-algebra of the complex which gives the formality.

\subsection{\texorpdfstring{Quadratic Lie algebras}
	       {Quadratic Lie algebras}}

Recall that a symmetric bilinear form 
$\kappa:\mathfrak{g}\times \mathfrak{g}\to\korps$ is called
\emph{invariant} if for all $\xi,\xi',\xi''\in\mathfrak{g}$
we have
\begin{equation}
    \kappa\big([\xi,\xi'],\xi''\big)
    =\kappa\big(\xi,[\xi',\xi'']\big).
\end{equation}
A triple $(\mathfrak{g},[~,~],\kappa)$ is called a \emph{quadratic Lie algebra} if the symmetric bilinear form $\kappa$ is
\emph{invariant and nondegenerate}. Examples are abelian Lie algebras (with any nondegenerate symmetric bilinear form), or
semisimple Lie algebra equipped with their \emph{Killing form}
$\kappa(\xi,\xi')=
\mathrm{trace}(\mathrm{ad}_\xi\circ\mathrm{ad}_{\xi'})$.
See e.g.~\cite{MR85}, \cite{AB93}, or \cite{Bor97} for more details on these algebras. We shall call any Lie algebra
admitting a nondegenerate invariant symmetric bilinear form
metrisable, see \cite{Bor97}. Let $q\in {\Sym}^2\mathfrak{g}$ be the `inverse' of 
$\kappa$: if $\kappa^\flat:\mathfrak{g}\to \mathfrak{g}^*$ denotes the canonical map 
$\xi\mapsto (\eta\mapsto \kappa(\xi,\eta))$, take its inverse $\kappa^\sharp:\mathfrak{g}^*\to \mathfrak{g}$, and consider it as an element $q$ in $\Sym[2]\mathfrak{g}$, or using a base $e_1,\ldots,e_n$ of $\mathfrak{g}$, and
$q=\sum_{i,j=1}^nq^{ij}e_i\otimes e_j$ where $q^{ij}\in\korps$
are the components of the symmetric bivector $q$, then
for all $\xi\in\mathfrak{g}$: $\sum_{i,j=1}^n\kappa(\xi,e_i)q^{ij}e_j=\xi$. Then $q$ is invariant under the adjoint representation of $\mathfrak{g}$.
Consider the morphism of commutative associative unital algebras $\korps[t]\to \Sym \mathfrak{g}$ induced by $t\mapsto
q$. Since every symmetric power $q^{\bullet m}\in 
\Sym[2m]\mathfrak{g}$ is easily seen to be nonzero (because
the free trivially graded commutative algebra 
$\Sym \mathfrak{g}$ does not have nilpotent elements), and
since the subspaces $\Sym[i]\mathfrak{g}$ are independent,
the above morphism is injective, and we denote its image
by $\korps[q]\subset \Sym \mathfrak{g}$. Clearly
every polynomial $\alpha\in\korps[q]\subset$ is invariant.  Next, there are three more important elements of $\ChEcochains(\mathfrak{g},\Sym\mathfrak{g})$,
the linear Poisson structure $\pi$ (see eqn \ref{EqDefLinearPoissonStructure}), the \emph{Euler field}
$E\in \mathfrak{g}\otimes \mathfrak{g}^*\cong \Hom(\mathfrak{g},\mathfrak{g})$, defined by the identity
map $\mathfrak{g}\to \mathfrak{g}$, 
$E=\sum_{i=1}^ne_i\otimes \epsilon^i$, 
and the \emph{Cartan $3$-cocycle}
$\Omega\in \Lambda^3\mathfrak{g}^*$ defined by
\begin{equation}
  \Omega(\xi,\xi',\xi'') = \kappa\big(\xi,[\xi',\xi'']\big)
\end{equation}
 for all $\xi,\xi',\xi''\in\mathfrak{g}$. Upon using formula
 (\ref{EqDefSchoutenBracketFinDim}) we easily compute the
 following Schouten brackets where 
 $\alpha,\beta,\gamma\in \korps[q]$ and $\alpha'$ denotes
 the derivative of the polynomial $\alpha$:  
 \begin{eqnarray}
  [\alpha,\beta]_s & = & 0, \label{EqCompSchoutenOneOne}\\
  \delta_{\mathfrak{g}}(\alpha) & = & [\pi,\alpha]_s = 0, 
        \label{EqCompSchoutenOnePi}  \\
  ~[E,\alpha]_s & = & 2q\alpha', 
    \label{EqCompSchoutenEOne} \\
  \delta_{\mathfrak{g}}(\alpha \wedge E) & = &
       [\pi,\alpha \wedge E]_s = \alpha \wedge\pi,  
       \label{EqCompSchoutenPiAlphaE}\\
  \delta_{\mathfrak{g}}(\alpha \wedge\Omega) & = & 
      [\pi, \alpha \wedge\Omega]_s = 0, 
      \label{EqCompSchoutenPiAlphaOmega}\\
  ~[E,\Omega]_s & = & -3\Omega,  
      \label{EqCompSchoutenEOmega}\\
  ~[\beta\wedge\Omega,\alpha]_s & = &
   2(\beta\alpha')\wedge\pi
   = \delta_{\mathfrak{g}}\big(2(\beta\alpha')\wedge E\big),
    \label{EqCompSchoutenBetaOmegaAlphaOne}\\
  ~[\beta\wedge \Omega,\gamma\wedge \Omega]_s 
     & = & 2(\beta\gamma'-\gamma\beta')\wedge\pi\wedge\Omega
           \nonumber \\
            & = &
    \delta_{\mathfrak{g}}\big(2(\beta\gamma'-\gamma\beta')\wedge
                        E\wedge\Omega\big).
         \label{EqCompSchoutenBetaOmegaGammaOmega}   
 \end{eqnarray}

 We can now compute a representing graded $3$-cocycle
 $z_3$ (see eqn (\ref{EqDefZThree}))
 for the characteristic $3$-class $c_3$ of
 $\ChEcochains(\mathfrak{g},\Sym \mathfrak{g})$ on certain
 elements of $\ChEcochains(\mathfrak{g},\Sym \mathfrak{g})$.
 For this purpose it seems to be interesting to single out
 a subclass of finite-dimensional quadratic Lie algebras: 
 We call a quadratic
 Lie algebra $(\mathfrak{g},[~,~],\kappa)$
 a \emph{Cartan-$3$-regular quadratic Lie algebra} if the cohomology
 class of the Cartan cocycle $\Omega$, $[\Omega]$, is nonzero.
 A metrisable Lie algebra will be called Cartan-$3$-regular
 if there is a nondegenerate symmetric invariant bilinear form
 $\kappa$ such that $(\mathfrak{g},[~,~],\kappa)$
 is Cartan-$3$-regular. Semisimple Lie algebras are well-known to be Cartan-$3$-regular. 
 
 \begin{lemma} \label{LPropertiesCartan3Regular}
 	Let $(\mathfrak{g},[~,~],\kappa)$ be a quadratic
 	Lie algebra of finite dimension, and let
 	$(\Sym \mathfrak{g})^\mathfrak{g}$ denote the subspace
 	of all $\mathrm{ad}$-invariant elements of 
 	$\Sym \mathfrak{g})^\mathfrak{g}$
 	\begin{enumerate}
 		\item Then $\korps[q]\subset 
 		(\Sym \mathfrak{g})^\mathfrak{g}
 		   \cong \ChE[0](\mathfrak{g},\Sym \mathfrak{g})$.
 		\item $(\mathfrak{g},[~,~],\kappa)$ is 
 		 Cartan-$3$-regular iff
 		   there is no derivation of $\mathfrak{g}$ whose
 		   $\kappa$-symmetric part is a nonzero multiple of the
 		   identity.
 		\item If $(\mathfrak{g},[~,~],\kappa)$
 		 is Cartan-$3$-regular then the linear 
 		 map $\korps[q]\to
 		 \ChE[3](\mathfrak{g},\Sym \mathfrak{g})$ defined
 		 by  $\alpha\mapsto [\alpha\wedge\Omega]$ is an
 		 injection.
 	\end{enumerate}
 	\begin{proof}
 		\textbf{1.} The last isomorphy is true for any
 		Lie algebra since the $0$-coboundaries vanish.
 		As $q$ is $\mathrm{ad}$-invariant, every polynomial
 		of $q$ is also $\mathrm{ad}$-invariant.\\
 		\textbf{2.} $(\mathfrak{g},[~,~],\kappa)$ is not
 		 Cartan-$3$-regular iff there is a $2$-form 
 		 $\theta:\Lambda^2\mathfrak{g}\to\korps$ such that
 		 $\Omega = -\delta_\mathfrak{g}\theta$. The space
 		 of all $2$-forms is isomorphic to the space of
 		 all $\kappa$-antisymmetric linear maps 
 		 $C:\mathfrak{g}\to \mathfrak{g}$ via 
 		 $C\mapsto \big((\xi,\eta)\mapsto
 		  \kappa(C(\xi),\eta)\big)$. Thanks to nondegeneracy
 		  and invariance of $\kappa$ the condition 
 		  $\Omega=-\delta_\mathfrak{g}\theta$ is easily be computed to be equivalent to
 		  \[
 		     [\xi,\eta] = C[\xi,\eta]-[C(\xi),\eta]
 		                     -[\xi,C(\eta)]
 		  \]
 		  for all $\xi,\eta\in\mathfrak{g}$. The above equation
 		  is equivalent to $D=C+I$ being a derivation of the
 		  Lie algebra $\mathfrak{g}$.
 		  The $\kappa$-symmetric part of $D$ is clearly 
 		  the identity map $I$.\\
 		  \textbf{3.} Note that for any nonnegative
 		  integer $n$ the symmetric power 
 		  $\kappa^n$ can be seen as a nonzero linear form
 		  on $\Sym[2n]\mathfrak{g}$ where in particular
 		  $\kappa^n(q^n)\neq 0$. Thanks to the invariance of
 		  $\kappa$ we have $\kappa^n([\xi,T])=0$ for any
 		  $\xi\in \mathfrak{g}$ and 
 		  $T\in\Sym[2n-1]\mathfrak{g}$. Denoting by $K_{2n}$ the linear map $\Sym[2n]\mathfrak{g}\otimes
 		  \Lambda \mathfrak{g}^*\to \Lambda \mathfrak{g}^*$
 		  sending $S\otimes \xi$ to $\kappa^n(S)\xi$ 
 		  and by $K$ the sum over all even degrees
 		  (on the odd degrees $K$ being defined to be zero)
 		  we see that $K$ intertwines Chevalley-Eilenberg differentials w.r.t.~the usual representation
 		  on $\Sym \mathfrak{g}$ induced by the adjoint representation and those w.r.t.~the 
 		  trivial representation on
 		  $\korps$. It suffices to show that each $q^n\wedge\Omega$ gives a non-zero class in cohomology. If there was $\theta\in 
 		  \Sym \mathfrak{g}\otimes \Lambda^2\mathfrak{g}^*$ such that $q^n\wedge\Omega=
 		  \delta_\mathfrak{g}(\theta)$, then --upon applying
 		  $K$ to this equation-- a non-zero multiple of
 		  $\Omega$ would be an exact form which would be in
 		  contradiction with $\mathfrak{g}$ being
 		  Cartan-$3$-regular.
 	\end{proof}
 \end{lemma}
  We have the following central result:
 \begin{theorem}\label{PQuadraticLieCharacteristicClass}
 	Let $\big(\mathfrak{g},[~,~],\kappa\big)$ be a finite-dimensional Cartan-$3$-regular quadratic Lie algebra.\\
 	Then the Hochschild complex of its universal envelopping
 	algebra is NOT $L_\infty$-formal.
\end{theorem}
\begin{proof}
	We shall show that the characteristic $3$-class $c_3$
	is nontrivial:\\
 	Let
 	$\alpha,\beta,\gamma\in \korps[q]$, upon writing 
 	$[\alpha]$ or $[\beta\wedge\Omega]$
 	for the $\delta_\mathfrak{g}$-cohomology classes represented by $\alpha$
 	and $\beta\wedge\Omega$, respectively.
 	From the Schouten brackets in 
 	(\ref{EqCompSchoutenOneOne}), (\ref{EqCompSchoutenBetaOmegaAlphaOne}), and
 	(\ref{EqCompSchoutenBetaOmegaGammaOmega}) which all give
 	$\delta$-coboundaries it is clear that the following
    graded Lie brackets in cohomology
 		 vanish: 
   \begin{equation}
 		 \big[[\alpha],[\beta]\big]_H=0,~~
 		 \big[[\alpha],[\beta\wedge\Omega]\big]_H=0,~~
 		 \big[[\beta\wedge\Omega],
 		 [\gamma\wedge\Omega]\big]_H=0.
 	\end{equation}
 	Next, we choose any graded vector space complement of the $\delta_\mathfrak{g}$-coboundaries
 	in the $\delta_\mathfrak{g}$-cocycles which includes 
 	the space of all $\alpha\in
 	\korps[q]$ and all
 	$\beta\wedge\Omega$, we get the resulting section $\phi_1:\mathfrak{H}\to \mathfrak{G}$ satisfying the natural condition $\phi_1([\alpha])=\alpha$
 	and $\phi_1([\alpha\wedge \Omega])=\alpha\wedge\Omega$
 	for all $\alpha\in \korps[q]
 	\subset (\Sym \mathfrak{g})^\mathfrak{g}$.
 	Then, according to eqn 
 	(\ref{EqCompShiftedLInfinityOrder2}),  eqs (\ref{EqCompSchoutenOneOne}), (\ref{EqCompSchoutenBetaOmegaAlphaOne}), and
 	(\ref{EqCompSchoutenBetaOmegaGammaOmega}) we can choose a
 	$\korps$-linear map $\phi_2:\Lambda^2\mathfrak{H}\to \mathfrak{G}$ of degree $-1$ satisfying
 	\begin{equation}
 	\phi_2(\alpha,\beta)=0~~\mathrm{and}~~~
 	\phi_2(\alpha,\beta\wedge\Omega)=2(\alpha'\beta)\wedge
 	E.
 	\end{equation} 
 	For later use we also note the following fact which will not be necessary in this proof:
 	\begin{equation}
 	\phi_2(\beta\wedge\Omega,\gamma\wedge\Omega) 
 	=2(\beta\gamma'-\gamma\beta')\wedge E\wedge \Omega.
 	\end{equation}
 	It follows that
 	the graded Chevalley-Eilenberg 
 	$\delta_\mathfrak{H}$-$3$-cocycle $z_3$ (which represents the characteristic $3$-class $c_3$ of 
 	the differential graded Lie algebra
 	$\mathfrak{G}=\ChEcochains(\mathfrak{g},\Sym \mathfrak{g})[1]$ and depends on $\phi_1$ and $\phi_2$,
 	see eqs
 	(\ref{EqDefWThree}) and (\ref{EqDefZThree})
 	takes the following values:
 	$z_3\big([\alpha],[\beta],
 	[\gamma]\big)=0$,  and, most importantly,
 	\begin{equation}
 	z_3\big([\alpha],[\beta],
 	[\gamma\wedge\Omega]\big) 
 	   =  8 [q\alpha'\beta'\gamma].
 	  \label{EqCompCharacteristicClassQuadratic003}
 	\end{equation}
 	Again, for later use and not necessary for this proof
 	we note that
 	\begin{equation}
 		z_3\big([\alpha],[\beta\wedge\Omega],
 		[\gamma\wedge\Omega]\big)  = 
 		-8\big[(q\alpha'(\beta\gamma'-\gamma\beta'))
 		\wedge\Omega\big].
 		\label{EqCompCharacteristicClassQuadratic033}
 	\end{equation}
 	 Finally, in case $c_3$ vanished there would be a graded
 $2$-form $\theta:\Lambda^2\mathfrak{H}\to \mathfrak{H}$
 (where $\mathfrak{H}=\ChE(\mathfrak{g},\Sym \mathfrak{g})[1]$)
 of degree $-1$ (since $z_3$ is of degree $-1$) such that
 $z_3=\delta_{\mathfrak{H}}\theta$. We evaluate 
$\delta_{\mathfrak{H}}\theta$ on the three elements
 $[\alpha],[\beta]$, and $[\gamma\wedge \Omega]$ of
 $\mathfrak{H}$. According to formula
 (\ref{EqCompGradedChevalleyEilenberg}) we need to know
 $\theta([\alpha],[\beta])$ (which must vanish since both
 $[\alpha]$ and $[\beta]$ are of degree $-1$ as is $\theta$)
 and $\theta([\alpha],[\gamma\wedge\Omega])$ which has to be of
 degree $0$, hence in $\ChE[1](\mathfrak{g},
 \Sym \mathfrak{g})[1]$. We consider the particular case $\alpha=q=\beta$ and $\gamma=1$. Let 
 $D\in \Hom(\mathfrak{g},\Sym \mathfrak{g})$ be a $\delta_\mathfrak{g}$-$1$-cocycle such that
 $[D]=\theta([q],[\Omega])$. Then the equation $z_3([q],[q],[\Omega])=
 (\delta_{\mathfrak{H}}\theta)([q],[q],[\Omega])$ would give
 (using the above eqn
  (\ref{EqCompCharacteristicClassQuadratic003}) and formula
  (\ref{EqCompGradedChevalleyEilenberg}))
  \begin{equation}
  8[q]  =  -2\big[ [q],[D]\big]_H =2[D(q)],
  \end{equation}
  since the Schouten bracket of a vector field with a function,
  $[f,X]_S$, equals $-X(f)$, and since $[\pi,q]_s=0$ we get
  $\big[[\pi,f]_s,q\big]_s=0$ for all $f\in\Sym \mathfrak{g}$
  showing that the last term in the above equation is 
  well-defined on the class $[D]$. Hence we would get the 
  equation 
  \begin{equation}
   D(q)=4q.
  \end{equation}
  Write $D=\sum_{r=0}^ND_r$ where for 
  each $r\in\mathbb{N}$ the component 
  $D_r\in \Hom (\mathfrak{g},\Sym[r]\mathfrak{g})$.
  Clearly, each $D_r$ must be
   a $\delta_\mathfrak{g}$-$1$-cocycle, hence $D_1:\mathfrak{g}\to\mathfrak{g}$ would be a derivation of the
   Lie algebra $\mathfrak{g}$, and comparing symmetric
   degrees we must have $D_1(q)=4q$. Elementary linear algebra
   (e.g.~expressing the previous equation in coordinates w.r.t.~a chosen base of $\mathfrak{g}$)
   gives for all $\xi,\xi'\in\mathfrak{g}$ the equation
   \[
        \kappa(D_1(\xi),\xi')+\kappa(\xi,D_1(\xi'))
        =4\kappa(\xi,\xi').
   \]
   This would show that the derivation 
   $D_1:\mathfrak{g}\to \mathfrak{g}$ has a $\kappa$-symmetric
   part equal to $2$ times the identity which is in contradiction to the hypothesis of $\big(\mathfrak{g},[~,~],\kappa\big)$ being 
   Cartan-$3$-regular, see the second statement of the preceding
   Lemma \ref{LPropertiesCartan3Regular}.
   Hence $c_3$ is a nontrivial class whence there is
   no formality.   
\end{proof}

\noindent The subclass of all Cartan-$3$-regular quadratic Lie algebras
includes also non semisimple Lie algebras whose derivations
are all antisymmetric, see e.g.~\cite{AB93}, for which there is
NO formality according to the preceding Proposition
\ref{PQuadraticLieCharacteristicClass}.

\subsubsection{\texorpdfstring{Reductive Lie algebras}{Reductive Lie algebras}}

Let $\big(\mathfrak{g},[~,~]\big)$ be a finite-dimensional
\emph{reductive Lie algebra}: recall that such a Lie algebra decomposes
into a direct sum $\mathfrak{g}=\mathfrak{z}\oplus [\mathfrak{g},\mathfrak{g}]$ where $\mathfrak{z}$ is its centre and the derived ideal $\mathfrak{l}=[\mathfrak{g},\mathfrak{g}]$
is a semisimple
Lie algebra, see e.g.~\cite{Jac62} for definitions.
Recall that every reductive Lie algebra is quadratic:
pick any nondegenerate symmetric bilinear form on 
$\mathfrak{z}$ and the \emph{Killing form} 
$(\xi,\xi')\mapsto
\mathrm{trace}(\mathrm{ad}_\xi\circ\mathrm{ad}_{\xi'})$ on
$\mathfrak{l}$, and define the nondegenerate invariant 
symmetric bilinear form $\kappa$ to be the orthogonal sum
of the two preceding ones. Note that the Cartan $3$-cocycle
$\Omega$ of $\mathfrak{g}$ is given by
$\Omega(z_1+l_1,z_2+l_2,z_3+l_3)
=\Omega_\mathfrak{l}(l_1,l_2,l_3)$ for any $z_1,z_2,z_3\in\mathfrak{z}$ and $l_1,l_2,l_3\in
\mathfrak{l}$ where $\Omega_\mathfrak{l}$ is the Cartan $3$-cocycle of $\mathfrak{l}$ which is well-known to
be a nontrivial $3$-cocycle
for $\mathfrak{l}$ if $\mathfrak{l}\neq \{0\}$ since all
derivations of a semisimple Lie algebra are well-known to
be inner hence antisymmetric w.r.t.~the Killing form.
It clearly is also a nontrivial $3$-cocycle for $\mathfrak{g}$.
Hence $\big(\mathfrak{g},[~,~],\kappa\big)$ is
Cartan-$3$-regular, and according to Theorem
\ref{PQuadraticLieCharacteristicClass} we get
\begin{prop}\label{PNonabelialReductiveNonformality}
	Let $\mathfrak{g}$ be a nonabelian reductive Lie algebra.\\
	Then the Chevalley-Eilenberg complex of $\mathfrak{g}$
	with values in $\Sym\mathfrak{g}$ (and hence the
	Hochschild complex of its universal enveloping
	algebra) is NOT formal in the $L_\infty$ sense.
\end{prop}

\noindent In case $\mathfrak{g}$ is semisimple it can be
shown (by the Whitehead Lemma and some standard representation theory) that the induced graded Lie bracket on cohomology
vanishes.

\subsubsection{\texorpdfstring{Lie algebra $\mathfrak{so}(3)$}{Lie algebra so(3)}}

The smallest semisimple Lie algebra is the Lie algebra
 $\mathfrak{so}(3)$ (isomorphic to $\mathfrak{sl}(2,\korps)$)
 which can be spanned by a basis $e_1,e_2,e_3$ subject to
 the brackets
 \[
  [e_1,e_2]=e_3,~~[e_2,e_3]=e_1,~~[e_3,e_1]=e_1,
 \]
 where all other brackets are clear from antisymmetry.
 The Killing form is given by $\kappa(e_i,e_j)=-2\delta_{ij}$.
 From the Whhitehead Lemma it is clear that
 \begin{equation}
 \mathfrak{H}=
 \ChE(\mathfrak{g},\Sym \mathfrak{g}) \cong \korps[q] \un \oplus \{0\} \oplus \{0\} \oplus \korps[q] \Omega
 \end{equation}
 where $\Omega$ is the Cartan $3$-cocycle.

As in the general semisimple case, the cohomology
$\ChE(\mathfrak{g},\Sym \mathfrak{g})$ does not inject
in the Chevalley-Eilenberg complex 
$\mathfrak{G}=\ChEcochains (\mathfrak{g,\Sym \mathfrak{g}})$
as a graded Lie subalgebra,
but we can define a smaller graded Lie subalgebra of $\mathfrak{G}$ which contains the cohomology, viz
\begin{equation}
 \mathfrak{G}_\text{red}\coloneqq \korps[q] \un \oplus \korps[q] E \oplus \korps[q] \pi \oplus \korps[q] \Omega
\end{equation}
where $E$ is the Euler field, $\pi$ is the linear Poisson structure.
\begin{prop}
	$\mathfrak{G}_\text{red}$  is a differential graded Lie subalgebra of $(\mathfrak{G},\delta,[~,~]_S)$, and there is a contraction
	\begin{equation} \label{so3-contraction}
	\begin{tikzpicture}[baseline=(current bounding box.center)]
	\matrix (m) [matrix of math nodes,nodes in empty cells,column sep=1em,text height=1.5ex,text depth=0.25ex]
	{\mathfrak{H} & & (\mathfrak{G}_{\text{red}},\delta) & \\};
	\path[right hook->]
	([yshift=-5pt]m-1-1.north east) edge node [above] {$i$} ([yshift=-5pt]m-1-3.north west);
	\path[->>]
	([yshift=5pt]m-1-3.south west) edge node [below] {$p$} ([yshift=5pt]m-1-1.south east);
	\draw[->] (m-1-4.north) arc (120:-120:2ex);
	\draw (m-1-4) ++(2em,0em) node {$h$};
	\end{tikzpicture}
	\end{equation}
	where $i$ is the natural injection, $p$ the natural projection (with kernel $\korps[q] E \oplus 
	\korps[q] \pi$ ), and the map
	$h$ is given by $h=h^1 : \korps[q] \pi \to \korps[q] E$, $h^1(\alpha \wedge\pi) = \alpha \wedge E$, for $\alpha \in \korps[q]$, and is defined to vanish in degree $-1,0,2$.\\
	The injection $\mathfrak{G}_\text{red}\to \mathfrak{G}$
	is a quasi-isomorphism of differential graded Lie algebras.
\end{prop}
\begin{proof}
	This follows from the identities 
	(\ref{EqCompSchoutenOneOne}) --
	(\ref{EqCompSchoutenBetaOmegaGammaOmega}).
\end{proof}


For this simple example we can use the $L_\infty$-Perturbation Lemma to compute an explict 
$L_\infty$-structure on the cohomology $\mathfrak{H}$ and an explicit $L_\infty$-quasi-isomorphism $e^{*\varphi}$ from
$\Sym (\mathfrak{H}[1])\to \Sym (\mathfrak{G}[1])$ with
$\varphi=\varphi_1+\varphi_2$ where $\varphi_1=i[1]$:

\begin{theorem} With the above notation we have the following:
	\begin{enumerate}
		\item 	The Chevalley-Eilenberg complex of $\mathfrak{so}(3)$ is NOT formal.
		\item There is a $L_\infty$ structure $d$ on $\Sym (\mathfrak{H}[1])$ whose only
		nonvanishing Taylor coefficient is $d_3$ (which can
		be given by 
		the shifted characteristic $3$-cocycle $z_3$, see eqn
		(\ref{EqDefZThree}), \ie $d_3=z_3[-1]$) for its
		only nonvanishing component (up to permutations).
		\item There is a $L_\infty$-quasi-isomorphism
		$e^{*\varphi}$ from 
		$\big(\Sym (\mathfrak{H}[1]), \overline{d_3}\big)$
		to \\
		$\big(\Sym (\mathfrak{G}_{\text{red}}[1]),
		\overline{\delta[1]}+ \overline{D_2}\big)$ (where
		$D_2$ denotes the shifted Schouten bracket).
		The only nonvanishing Taylor coefficients 
		of $e^{*\varphi}$ are
		$\varphi_1=i[1]$ and $\varphi_2$ which can explicitly
		 be given.
	\end{enumerate}
\end{theorem}

\begin{proof}
	\textbf{1.} This is a particular case of Proposition
	\ref{PNonabelialReductiveNonformality}.\\
	\textbf{2.} and \textbf{3.}
	We compute the formulas from the
	$L_\infty$-Perturbation Lemma, see eqs
	(\ref{EqDefInducedDifferentialHPTLInf}) and (\ref{EqDefLInfinityPertLemmaInjection}) which we give in terms of the geometric series:
	\begin{eqnarray*}
	   d & = & 
	       \sum_{r=0}^\infty\mathrm{pr}_{\mathfrak{H}[1]}
	       \circ \overline{D_2}\circ
	       (-\eta\circ\overline{D_2})^r\circ e^{*\varphi_1}
	       ,\\
	   \varphi & = &  \sum_{r=0}^\infty
	      \mathrm{pr}_{\mathfrak{G}[1]} 
	       \circ (-\eta\circ\overline{D_2})^r\circ e^{*\varphi_1}.
	\end{eqnarray*}
	In order to understand 
	--for any nonnegative integer $r$--
	the iterated product 
	$(-\eta\circ \overline{D_2})^r$, we shall apply
	$-\eta\circ \overline{D_2}$
	to a graded symmetric word containing $k+m+l$ letters or terms of the following kind: $k$ times a term
	of degree $-2$, i.e.~of the form $\alpha\in \korps[q]$,
	$l$ times a term of degree $1$, i.e. of the form
	$\beta\wedge\Omega\in \korps[q]\wedge \Omega$, and
	$m$ times a term of degree $-1$, i.e.of the form
	$\gamma\wedge E\in \korps[q]\wedge E$: the application
	of the
	shifted Schouten bracket $\overline{D_2}$
	will produce two sums of linear combinations of graded symmetric words; the first type of words containing $k-1$ terms of degree
	$-2$, $m$ terms of degree $-1$, one term of degree $0$ proportional to $\tilde{\gamma}\wedge\pi$ with $\tilde{\gamma}\in \korps[q]$
	(which comes from the Schouten bracket of a degree $-2$ term and a degree $1$-term), and
	$l-1$ terms of degree $1$; the second type of words containing
	$k$ terms of degree $-2$, $m-1$ terms of degree $-1$, and $l$ terms of degree $1$ (which come from Schouten brackets
	involving at least one Euler field). An ensuing application of $\eta$
	shifts the terms proportional to $\pi$ (in the first sum)
	to a term proportional
	to $E$ and kills the second sum (since it obviously is in the kernel
	of the graded biderivation $\overline{h}$).
	As a result we get a linear combination of graded symmetric words containing $k-1$ terms of degree $-2$, $m+1$ terms of
	degree $-1$, and $l-1$ terms of degree $1$.
	By induction, and $r$-fold iteration yields words
	with $k-r$ terms of degree $-2$, $m+r$ terms of degree $-1$,
	and $l-r$ terms of degree $l-r$.\\
	In the above formulas for $d$ and $\varphi$ we have $m=0$ since the expressions are applied to words containing letters in the cohomology. It follows that for all integers
	$r\geq 2$ there will be $r\geq 2$ factors of type
	$\tilde{\gamma}\wedge E$: application of the projection
	$\mathrm{pr}_{\mathfrak{G}[1]}$ will kill these terms
	because there are at least two factors. It follows that
	there are only two surviving Taylor coefficients of
	$\varphi$: $\varphi_1$ (the case $r=0$) and 
	$\varphi_2$ (the case $r=1$). 
	Computing on arguments $\alpha \un, \beta \omega$ in $\mathfrak{G}[1]$, with $\alpha,\beta \in \korps[q]$, we obtain
	\begin{equation*}
	\varphi_2(\alpha \un,\beta \un) = 0, \qquad \varphi_2(\alpha \un,\beta \omega) = \alpha' \beta E \qquad \text{and} \qquad \varphi_2(\alpha \omega,\beta \omega) = 0.
	\end{equation*}
	On the other hand, for each integer $r\geq 2$ an application
	of the shifted Schouten bracket $\overline{D_2}$ will leave
	at least one factor of the type $\tilde{\gamma}\wedge E$
	which is in the kernel of the projection to cohomology,
	$\mathrm{pr}_{\mathfrak{H}[1]}$. It follows that all
	Taylor coefficients
	$d_{r+2}$ of $d$ vanish for $r\geq 2$, and the shifted induced bracket on cohomology, $d_2$ (the case $r=0$),
	vanishes thanks to fact that the induced graded Lie
	bracket on cohomology vanishes for semisimple Lie algebras 
	. Hence the only remaining
	Taylor coefficient is $d_3$ (the case $r=1$) which is of the
	form
	\begin{eqnarray*}
		d_3(\alpha\un,\beta\un,\gamma\wedge\Omega)
		& = & 8q\alpha'\beta'\gamma,\\
	    d_3(\alpha\un,\beta\wedge \Omega,\gamma\wedge\Omega)
	    &=&-8(q\alpha'(\beta\gamma'-\gamma\beta'))\wedge\Omega.
	\end{eqnarray*}
\end{proof}

\subsection{Heisenberg algebra}

We consider the three-dimensional Heisenberg Lie algebra
whose underlying vector space is $\korps^3$ with basis
$x,y,z$, and the only nonvanishing bracket is given by
$[x,y]=z=-[y,x]$.
or, writing the Lie bracket as a bivector $[~,~] = \pi = z \partial_x \wedge \partial_y \in \Hom\big(\Lambda^2 \mathfrak{g},\mathfrak{g}\big)$. Although the Lie bracket is simpler than the one of $\mathfrak{so}(3)$, the cohomology is more complex, and is not abelian. 

In \cite{E14}, the cohomology of the Chevalley-Eilenberg complex $\ChEcochains(\mathfrak{g},\Sym\mathfrak{g})$ has been computed and shown that it is not formal. 

\section{Free Lie algebra} \label{Sec:Free-Lie-alg}

We shall closely follow Chapitre 3 of the thesis \cite{El12}.

Let $V$ be a vector space over $\korps$ and $\mathcal{L}V$ the associated free Lie algebra. Then its universal enveloping algebra is well-known to be
isomorphic to the free associative algebra
\begin{equation*}
\UEA(\mathcal{L}V) \cong \Tens V.
\end{equation*}
Here the Hochschild cohomology can be computed using a free resolution, see e.g.~\cite[Chap.IX p.181]{CE56}, and is only concentrated in degree $0$ and $1$, composed of its centre 
$\Tens V^{\Tens V}$ and the space of all outer derivations:
\begin{align*}
\Hoch(\Tens V,\Tens V) & {}\cong \Tens V^{\Tens V} \oplus \Der(\Tens V,\Tens V)/\operatorname{Inder}(\Tens V,\Tens V) 
\end{align*}
where, as usual, $\Der(\Tens V,\Tens V )$ denotes the Lie
algebra of all derivations of the algebra $\Tens V$, and $\operatorname{Inder}(\Tens V,\Tens V)$
denotes the space of all inner derivations, which are adjoint
representations $\mathrm{ad}_x:\Tens V\to\Tens V$ for all
$x \in \Tens V$ defined by
$\mathrm{ad}_x(y)=xy-yx$ for all $y \in \Tens V$.
Note that $\mathrm{ad}_x=b(x)$ for the Hochschild coboundary $b$. We shall sometimes denote the quotient Lie algebra
$\Der(\Tens V,\Tens V)/\operatorname{Inder}(\Tens V,\Tens V)$ 
by $\mathfrak{outder}$, and shall
again write $\mathfrak{G}$ for the graded Lie algebra
$\big(C_H(\Tens V,\Tens V)[1],[~,~]_G,b\big)$.\\

For $V$ of \textbf{dimension 0}, $\Tens \{0\} \cong \korps$, the centre is isomorphic to the field $\Tens V^{\Tens V} \cong \korps$, and \textbf{there is a formality map} since the Hochschild cohomology injects as a graded abelian Lie subalgebra in the Hochschild complex
\begin{equation*}
\varphi_1 = id_\korps : \korps \to \bigoplus_{k \in\nat} \Hom(\korps^{\otimes k},\korps) \cong \bigoplus_{k \in\nat} \korps, \qquad \varphi_k = 0\ \text{for $k\geqslant 2$}.
\end{equation*}

For $V$ of \textbf{dimension 1} we can write $V=\korps e$ (having fixed a base vector $e$ of $V$), and $\Tens(\korps e) \cong \korps[x]$ is the commutative ring of polynomials in one variable, hence it is also equal to the center $\Tens V^{\Tens V} = \korps[x]$, and all inner derivations vanish. The space $\Hom(\korps e,\korps[x])=\{f \partial_x~|~ f \in \korps[x]\}$ is the Lie algebra of vector fields. \textbf{There is} again 
\textbf{a formality map} induced by the inclusion $\varphi_1$ of the Hochschild cohomology into the Hochschild complex as a graded Lie subalgebra, $\varphi_k = 0$ for $k\geqslant 2$. The map $\varphi_1$ is the identity on the center, and associates to $f \partial_x$ its derivation.

The truly interesting and more involved case is of course
given by $V$ of \textbf{dimension} 
$\mathbf{\geqslant 2}$. The zeroth cohomology group, i.e.~the
centre of $\Tens V$, is then well-known to be reduced to $\Tens V^{\Tens V} = \korps \un$. The graded Lie bracket $[~,~]_H$
on the cohomology $\mathfrak{H}=\korps \un \oplus \mathfrak{outder}$ is readily
computed by
\begin{equation}
  \big[(\lambda\un, D),(\lambda'\un, D')\big]_H
     = \big(0,[D,D']\big)
\end{equation}
for any $\lambda,\lambda'\in\korps$ and $D,D'\in \mathfrak{outder}$. Since the cohomology graded Lie algebra
is concentrated in degree $-1$ and $0$, its shift 
$\mathfrak{H}[1]$ is concentrated in degree $-2$ and $-1$.
Counting degrees we immediately get the following
\begin{theorem}\label{TFreeSigma}
	There is a $L_\infty$-structure $d$ on
	$\Sym (\mathfrak{H}[1])$ whose Taylor coefficients
	$d_n$ vanish for all integers $n\geq 4$,
	and $d_2=[~,~]_H[1]$.
	Moreover there is a $L_\infty$-quasi-isomorphism 
	$e^{*\varphi}$ from 
	$\big(\Sym (\mathfrak{H}[1]),
	\overline{d_2}+\overline{d_3}\big)$
	to  $\big(\Sym(\mathfrak{G}[1],\overline{b[1]}+
	\overline{D_2}\big)$ whose Taylor coefficients
	$\varphi_n$ vanish for all $n\geq 3$.\\
	 Finally, the map $d_3[-1]: \Lambda^3(\mathfrak{H}) 
	 \to \mathfrak{H}$ (which is of degree $-1$) can be 
	 seen as a
	 a scalar $3$-cocycle $\sigma$ of the Chevalley-Eilenberg cohomology $\ChEcochains\big(\mathfrak{outder},\korps\big)$ of the
	 Lie algebra of all outer derivations of $\Tens V$.
\end{theorem}

\begin{proof}
	The $L_\infty$ structure $d=\sum_{n=2}^\infty d_n$ 
	on $\Sym (\mathfrak{H}[1])$ and the $L_\infty$-map
	$e^{*\varphi}$
	exist by the general arguments given in the preceding 
	sections, for instance thanks to the fact that we
	can always find a chain homotopy to relate
	$\mathfrak{H}$ and $(\mathfrak{G},b)$  in context of a deformation retract and using Theorem \ref{Thm:PertBordemann}.\\
	Since $d$ is of degree $1$, and $\varphi$ is of degree
	$0$ we get for all integers
	$k\geq 1$ and $i_1,\ldots,i_k\in\{-2,-1\}$ that
	$d_k(\mathfrak{H}[1]^{i_1}\bullet\cdots\bullet
	\mathfrak{H}[1]^{i_k}) \subset \mathfrak{G}_\text{red}[1]^{i_1+\dotsb+i_k+1}$
	and
	  $\varphi_k(\mathfrak{H}[1]^{i_1}\bullet\cdots\bullet
	    \mathfrak{H}[1]^{i_k}) \subset \mathfrak{G}_\text{red}[1]^{i_1+\dotsb+i_k}$. 
	    For $k\geq 4$ we have $i_1+\cdots+i_k+1\leq -3$ 
	    and for all $k\geq 3$ we have $i_1+\cdots+i_k\leq -3$,
	    whence $d_k=0$ for all $k\geq 4$ and $\varphi_k=0$
	    for all $k\geq 3$.\\
	    $d_3[-1]$ is a graded Chevalley-Eilenberg $3$-cocycle
	    of degree $-1$: since $\mathfrak{H}^{-1}\cong \korps \un$ is central, it follows that all other components
	     of $d_3[-1]$ are reduced to zero with the possible exception of the restriction of $d_3[-1]$ to
	     three arguments in $\mathfrak{outder}=\mathfrak{H}^0$
	     whose image is in $\mathfrak{H}^{-1}\cong \korps \un$.
	     This surviving component can be seen as an ungraded scalar $3$-cocycle $\sigma: 
	     \Lambda^3\mathfrak{outder}\to \korps$ of the ungraded Lie algebra $\mathfrak{outder}$.
\end{proof}

In order to check formality we have to check whether the 
aforementioned $3$-cocyle $\sigma$ can be a coboundary, and this requires
some more explicit computations:

The Hochschild $1$-cocycles of $\Tens V$ comprise the space of all derivations 
$\Der(\Tens V,\Tens V)$ of
$\Tens V$: since every derivation is uniquely determined by its
restriction to the space of generators $V$, and in turn
every linear map $\psi:V\to \Tens V$ uniquely extends to 
a derivation by the Leibniz rule there is a linear isomorphism
$\overline{(~)}: \Hom(V,\Tens V)\to 
\Der(\Tens V,\Tens V)\subset \Hom(\Tens V,\Tens V)$ defined by
\begin{equation}
 \overline{\psi}(\un)=0,~~~
\overline{\psi}(x_1\cdots x_n)=\sum_{r=1}^nx_1\cdots x_{r-1}
     \big(\psi(x_r)\big)x_{r+1}\cdots x_n.
\end{equation}
for all $x_1,\ldots,x_n\in V$.
We shall sometimes denote $\Hom(V,\Tens V)$ by
 $\mathfrak{der}$. As for the
coderivations we can pull-back the usual Lie bracket of derivations from $\Der(\Tens V,\Tens V)$ to a Lie bracket $[~,~]_D$ on the space $\mathfrak{der}$ by means of the linear isomorphism $\overline{(~)}$: for any 
$\psi, \chi \in \Hom(V,\Tens V)$ we compute
\begin{equation}
[\psi,\chi]_D \coloneqq \overline{\psi} \circ\chi - \overline{\chi}\circ\psi,
\end{equation}
and $\overline{(~)}$ is a morphism of Lie algebras, 
i.e.~$\overline{[\psi,\chi]_D}
=[\overline{\psi},\overline{\chi}]$.
Moreover we shall write $b':\Tens V\to \Hom(V,\Tens V)$
for the restriction of the adjoint representation
$b(x)=\mathrm{ad}_x$ to $V$ (for all $x\in\Tens V$).
We have $\overline{b'(x)}=b(x)$, hence 
$b'(xy-yx)=[b'(x),b'(y)]_D$.
It follows that $b'\Tens V$ is an ideal of the Lie algebra
$\big(\Hom(V,\Tens V),[~,~]_D\big)$.\\
The space $\Hom(V,\Tens V)$ carries an additional $\intg$-grading (called \emph{tensor grading})
according to the degree
\(
\Hom(V,\Tens V)^{(k)} = \Hom(V,V^{\otimes k+1}),
\)
for $k\geq -1$, and $\{0\}$ for $k\leq -2$.
The tensor grading is auxiliary, no signs are attached. The space $b'\Tens V^+$ also carries the degree of $\Tens V^+$, ${\Tens V^+}^{(k)} = V^{\otimes k}$ for $k \geqslant 1$. In the first degrees, we have $\Hom(V,\Tens V)^{(-1)} = \Hom(V,\korps) = V^\star$, $(b\Tens V^+)^{(-1)} = \{0\}$, and 
$\Hom(V,\Tens V)^{(0)} = \Hom(V,V)$, $(b'\Tens V^+)^{(0)} = \{0\}$. Note that brackets and $b'$ are of tensor degree
$0$, whence the cohomology $\korps\un\oplus \mathfrak{outder}$
is in addition graded by the tensor degree.

As for the Lie algebra $\mathfrak{so}(3)$ we can now define
a smaller differential graded Lie algebra 
$\mathfrak{G}_{\text{red}}$ which injects in the Hochschild complex of $\Tens V$ as a differential graded subalgebra, viz.~
\begin{equation}
   \mathfrak{G}_{\text{red}} 
   =\mathfrak{G}^{-1}_{\text{red}}\oplus
   \mathfrak{G}^{0}_{\text{red}}\coloneqq \Tens V \oplus \Hom(V,\Tens V) 
\end{equation}
equipped with the rather simple graded Lie bracket
(where $x,y\in \Tens V$, $\psi,\chi\in\mathfrak{der}$ 
and we write ordered pairs $(x,\psi)$ for elements of the
direct sum $\Tens V \oplus \Hom(V,\Tens V)$)
\begin{equation}\label{EqDefReducedSchoutenBracketFreeLie}
    \big[(x,\psi),(y,\chi)\big]_{\text{red}}:= 
    \big(\overline{\psi}(y)-\overline{\chi}(x),
       [\psi,\chi]_D\big),
\end{equation}
and differential
 $b_{\text{red}}(x,\psi)\coloneqq\big(0,b'(x)\big)$.
It is easy to see that the injection $(x,\psi)\mapsto
x+\overline{\psi}$ is a quasi-isomorphism of differential
graded Lie algebras $\big(\mathfrak{G}_{\text{red}},
[~,~]_{\text{red}},b_{\text{red}}\big)\to
\big(\mathfrak{G},[~,~]_G, b\big)$.

Next we would like to define a chain homotopy $h$ in
$\mathfrak{G}_{\text{red}}$. To this end we first choose
a complementary subspace $\mathcal{H}^0\subset \Hom(V,\Tens V)$
to the space of all coboundaries $b'\Tens V$, i.e.~restrictions of inner derivations, in $\Hom(V,\Tens V)$ in the following
way: we can suppose that it is graded (with respect to the tensor degree), i.e.~$\mathcal{H}^0 = 
\bigoplus_{n \geqslant -1} \mathcal{H}^{0(n)}$, and we set $\mathcal{H}^{0(-1)} = V^\star$ (dual space of $V$), $\mathcal{H}^{0(0)} = \Hom(V,V)$, and for $k\geqslant 1$, we choose in each $\Hom(V,V^{\otimes k+1})$ a complementary subspace $\mathcal{H}^{0(k)}$ to the inner derivations, \ie $\Hom(V,V^{\otimes k+1}) = \mathcal{H}^{0(k)} \oplus 
(b'\Tens V^+)^{(k)}$. For the component
of degree $-1$ of $\mathfrak{G}_{\text{red}}$, $\Tens V$,
we have the natural decomposition $\Tens V= \korps\un \oplus
\Tens V^+$, the latter being the augmentation ideal of
$\Tens V$, \ie the sum of all elements of strictly positive tensor degree. Hence we set $\mathcal{H}^{-1}=\Tens V^+$
which also carries a second grading according to tensor
degree, $\mathcal{H}^{-1(k)}={\Tens}^k V$ for all integers
$k\geq 1$.
For each integer $k\geq -1$
let $P_k : \Hom(V,V^{\otimes k+1}) \to \mathcal{H}^{0(k)}$ be the canonical projection having kernel $(b'\Tens V)^k$.
We set $P =\sum_{k \geqslant -1} P_k$.
Hence for each integer $k\geq -1$ the linear map
$id_{\mathfrak{der}^k} - P_k$ (which vanishes for $k=-1,0$) is a projection onto
$b'\Tens V$ which is in bijection with $\Tens V^+$ via 
$b'$. Using the inverse of this bijection there is,
for each integer
$k\geqslant -1$, a unique linear map 
$Q_k : \Hom(V,V^{\otimes k+1}) \to V^{\otimes k}$ be such that 
$id_{\mathfrak{der}^k} - P_k = bQ_k$. Setting
$Q_{-1} = 0$, $Q_0 = 0$, $Q= \sum_{k \geqslant 1} Q_k$, we define the chain homotopy
$h:\mathfrak{G}_{\text{red}}\to\mathfrak{G}_{\text{red}}$ 
\begin{equation}\label{EqDefFreeAlgebraHomotopyGeneral}
       h(x,\psi) = \big(Q(\psi),0\big)
\end{equation}
 for all 
$\psi\in\Hom(V,\Tens V)$ and $x\in\Tens V$.
Moreover, since the restriction of the natural projection $\Hom(V,\Tens V)\to
\mathfrak{outder}$ to the subspace $\mathcal{H}^{0}$ of
$\Hom(V,\Tens V)$ is a bijection, its inverse gives an injection $\mathfrak{outder}\to \mathcal{H}^{0}\subset \Hom (V,\Tens V)$ which, combined with the canonical injection
$\korps\un\to \Tens V$, defines an injection $i$ of the
cohomology $\mathfrak{H}$ into $\mathfrak{G}_{\text{red}}$.
On the other hand, the natural projection $\Hom(V,\Tens V)\to
\mathfrak{outder}$ combined with the canonical projection
$\Tens V\to \korps \un$ (having kernel $\Tens V^+$) defines
a surjection $p:\mathfrak{G}_{\text{red}}\to \mathfrak{H}$.
Obviously, $p\circ i=\mathrm{id}_\mathfrak{H}$ and
$i\circ p$ equals $P$ on $\mathfrak{G}^0_{\text{red}}
=\Hom(V,\Tens V)$ and the projection onto $\korps\un\subset
\Tens V=\mathfrak{G}^{-1}_{\text{red}}$. It follows
that there is a contraction
\begin{equation*}
\begin{tikzpicture}[baseline=(current bounding box.center)]
\matrix (m) [matrix of math nodes,nodes in empty cells,column sep=1em,text height=1.5ex,text depth=0.25ex]
{\mathfrak{H} & & (\mathfrak{G}_{\text{red}},b) & \\};
\path[right hook->]
([yshift=-5pt]m-1-1.north east) edge node [above] {$i$} ([yshift=-5pt]m-1-3.north west);
\path[->>]
([yshift=5pt]m-1-3.south west) edge node [below] {$p$} ([yshift=5pt]m-1-1.south east);
\draw[->] (m-1-4.north) arc (120:-120:2ex);
\draw (m-1-4) ++(2em,0em) node {$h$};
\end{tikzpicture}
\end{equation*}
The graded $\mathfrak{H}$-$3$-cocyle $\sigma_3$
(see the preceding Theorem \ref{TFreeSigma}) is surprisingly
simple:
\begin{prop}\label{PSigmaGeneralComplement}
 The graded $\mathfrak{H}$-$3$-cocyle $\sigma_3$ 
 from $\Lambda^3 \mathfrak{outder}\to \korps$ (defined
 in Theorem \ref{TFreeSigma}) is of tensor degree zero.
 For any $\alpha,\beta\in \mathfrak{outder}^{(-1)}=V^*$,
 $A,B,C\in \mathfrak{outder}^{(0)}=\Hom(V,V)$, $\rho\in
 \mathfrak{outder}^{(1)}$, and $\psi\in
 \mathfrak{outder}^{(2)}$ we get
 \begin{eqnarray}
   \sigma(A,B,C) & = & 0 , \label{EqCompSigma000}\\
    \sigma(\alpha, B, \rho) & = &
      -\alpha\big(
         Q_1([B,\rho]_D)\big),
        \label{EqCompSigma-101} \\
   \sigma(\alpha, \beta, \psi) & = &
   -\alpha\big(
   Q_1([\beta,\psi]_D)\big)+\beta\big(
   Q_1([\alpha,\psi]_D)\big),
   \label{EqCompSigma-1-12} 
 \end{eqnarray}
 whereas all other components of $\sigma$ (which are
 no permutations of the above) vanish. 
\end{prop}
\begin{proof}
	According to formula (\ref{EqCompPhiTwoPertShifted})
	and formula $w_3$ (see eqn (\ref{EqDefWThree}))
	we see that $\sigma$ is of tensor degree $0$,
	hence for any homogeneous $\psi_1,\psi_2,\psi_3\in
	\mathfrak{outder}$ it follows that
	 $\sigma(\psi_1,\psi_2,\psi_3)$ is of tensor degree zero
	 (since $\korps\un\subset \Tens V =\mathfrak{H}^{-1}$ is of 
	 tensor degree zero), hence $|\psi_1|+|\psi_2|+|\psi_3|=0$
	 leaving for the only possible non-zero components
	 three possibilities $|\psi_1|=0$, $|\psi_2|=0$,
	 $|\psi_3|=0$; then $|\psi_1|=-1$, $|\psi_2|=0$,
	 $|\psi_3|=1$; and finally $|\psi_1|=-1$, $|\psi_2|=-1$,
	 $|\psi_3|=2$. Identifying $\mathfrak{outder}$ with the
	 complement $\mathcal{H}^0$ in $\mathfrak{der}$
	 and the cohomological Lie bracket $[~,~]_H$ in
	 $\mathfrak{outder}$ with the projection
	 $(\mathrm{id}-Q)([~,~]_D)$
	 we get
	\begin{eqnarray}
       \lefteqn{\sigma(\psi_1,\psi_2,\psi_3) }\nonumber \\
       & = & -\epsilon\Big(\big[\psi_1,
         Q[\psi_2,\psi_3]_D
       \big]_{\mathfrak{G}_{red}}
       	-Q\big(\big[[\psi_1,\psi_2]_H,
       	\psi_3\big]_{\mathfrak{G}_{red}}\big)
       	+ \mathrm{cycl.}\Big) \nonumber\\
       & = & -\epsilon\Big(
       \overline{\psi_1}\big(
       Q([\psi_2,\psi_3]_D)\big)+ \mathrm{cycl.}\Big) 
             \label{EqCompSigmaFreeGeneral}      
	\end{eqnarray}
	where the last three terms vanish thanks to the fact that
	the double bracket is of tensor degree $0$, whence its
	result is of tensor degree $0$, and $Q_0=0$.\\
	Eqn (\ref{EqCompSigma000}) immediately follows since
	$Q_0=0$. Two of the three terms in (\ref{EqCompSigma-101})
	vanish since $Q_{-1}=0$ and $Q_0=0$, and $\overline{\alpha}$
	reduces to the application of $\alpha$ applied to the
	vector which is the value of $Q_1$. In eqn
	(\ref{EqCompSigma-1-12}) the term with the bracket
	$[\alpha,\beta]_D=0$ vanishes leaving the other two.
\end{proof}

We shall now construct a more explicit complement
$\mathcal{H}$ to the coboundaries in $\mathfrak{der}$
in the case where \textbf{$V$ is of finite dimension 
	$\mathbf{N}$}:
Let $e_1,\ldots,e_N$ be a base and $\epsilon^1,\ldots,\epsilon^N$ be the corresponding dual base.
For $n \in \nat$ we can write the applications $\psi\in \Hom(V,V^{\otimes n+1})$ as
\[
\psi = \sum_{j, i_0, \dotsc , i_n=1}^N \psi_j^{i_0 \dotso i_n} e_{i_0} \otimes \dotsb \otimes e_{i_n} \otimes \epsilon^j.
\]
where $\psi_j^{i_0 \dotso i_n}\in \korps$ are the components
of $\psi$ w.r.t.~the base. We have canonically identified
$\Hom(V,V^{\otimes n+1})$ with $V^{\otimes n+1}\otimes V^*$.
For each integer $n \geq -1$, we consider the following 
linear map
$S_n :  \Hom(V,V^{\otimes n+1})  \to {\Tens}^n V$ defined
by $ S_{-1}  =   0$, $S_0    =   0$, and for all $n\geq 1$,
$v_0,v_1,\ldots,v_n\in V$ and $\alpha\in V^*$ we set
$S_n(v_0\otimes\cdots\otimes v_n\otimes \alpha)$ equal to
$\alpha(v_0)v_1\otimes\cdots\otimes v_n$ which reads
in components
\begin{equation}
  S_n(\psi)  = 
  \sum_{j, i_1, \dotsc , i_n=1}^N \psi_j^{j i_1 \dotso i_n} e_{i_1} \otimes \dotsb \otimes e_{i_n}~~~\mathrm{if~}n\geq 1,
\end{equation}
and can be viewed
as a kind of \emph{`first factor trace'} for $n\geq 1$. Note that each $S_n$ is invariant under the general linear group of $V$.
We write
$S : \Hom(V,\Tens V) \to \Tens V$ for the sum $S \coloneqq \sum_{n \geqslant -1} S_n$ whence $S$ is homogenous of degree
$0$ w.r.t~the tensor grading. For each integer $n\geq 1$
denote by $\zeta_n:V^{\otimes n}\to V^{\otimes n}$ the linear map defined by the cyclic permutation where 
$v_1\otimes\cdots \otimes v_n$
is sent to $v_2\otimes\cdots\otimes v_n\otimes v_1$
for all $v_1,\ldots,v_n\in V$. Observing
that for each $a\in {\Tens}^nV$, $n\geq 1$, 
the inner derivation $b'(a)$ has components $(\mathrm{ad}_a)_j^{i_0i_1\cdots i_n}$ given by
$a^{i_0i_1\cdots i_{n-1}}\delta^{i_n}_j-
\delta^{i_0}_ja^{i_1\cdots i_n}$ we get
\begin{equation}
     S_n\big(b'(a)\big)=\zeta_n(a)-Na.
\end{equation}
Since obviously $\zeta_n^{\circ n}=\mathrm{id}_{{\Tens}^nV}$
it follows that $\zeta_n$ is diagonalizable, and 
the eigenvalues of $\zeta_n$ are in the set of all $n$th roots
of unity, whence $\zeta_n-N\mathrm{id}_{{\Tens}^nV}$ is invertible since $N\geq 2$. This shows that for each integer $n\geq 1$ the map $S_n$ is surjective, and the intersection of $b'({\Tens}^nV)$
with the kernel $\mathcal{H}^{0(n)}$ of $S_n$ is equal to $\{0\}$.
By elementary finite-dimensional linear algebra we conclude
that $\mathcal{H}= \bigoplus_{n \geqslant -1} \Kr S_n$ is a graded complement of $b'\Tens V^+$, and thus defines a section
$\mathfrak{outder}\to \mathcal{H}\subset \mathfrak{der}$.
Recall that $\mathcal{H}^{0{(-1)}}=V^*=\mathfrak{der}^{(-1)}$
and $\mathcal{H}^{0(0)}=\Hom(V,V)=\mathfrak{der}^{(0)}$. By inverting
$\zeta_n-\mathrm{id}_{{\Tens}^nV}$ we get the map 
$Q=\sum_{n\geq -1}Q_n:\mathfrak{der}\to \Tens V$ by setting
$Q_{-1}=0$, $Q_0=0$, and for each integer $n\geq 1$
\begin{equation}\label{EqDefQNFinDim}
    Q_n= -\frac{1}{N^n-1}\sum_{r=0}^{n-1}
          N^{n-r-1}\zeta_n^{\circ r}\circ S_n.
\end{equation}
\begin{prop}
	In the notation of Proposition
	\ref{PSigmaGeneralComplement} we have
	\begin{equation}\label{EqCompSigmaSpecial000And-101}
	 \sigma(A,B,C)=0,~~\sigma(\alpha,B,\rho)=0,~~~\mathrm{and}
	\end{equation}
	\begin{equation}\label{EqCompSigmaSpecial-1-12}
     \sigma(\alpha,\beta,\psi) =
     \frac{1}{N-1} \sum_{k,j,l} \alpha_l \beta_k 
       (\psi_j^{l j k} - \psi_j^{k j l}),
	\end{equation}
	and there are $\alpha,\beta,\psi$ such that
	$\sigma(\alpha,\beta,\psi)\neq 0$.
\end{prop}
\begin{proof}
 The first eqn in (\ref{EqCompSigmaSpecial000And-101})
 follows from (\ref{EqCompSigma000}), and the second one
 from eqn (\ref{EqCompSigma-101}), from the definition
 of $S_1$ (\ref{EqDefQNFinDim}) and from the fact that
 $S_1$ is invariant under the linear group whence
 $S_1([B,\rho]_D)=B\big(S_1(\rho)\big)=0$ since
 $\rho\in \mathcal{H}^{0(1)}=\mathrm{Ker}(S_1)$.
 The last eqn is a straight-forward computation using
 eqs (\ref{EqDefQNFinDim}) and (\ref{EqCompSigma-1-12}).
 Taking $\alpha=\epsilon^1$, $\beta=\epsilon^2$, and
 $\psi=e_1\otimes e_2\otimes e_2\otimes\epsilon^2$ we get
 $\sigma(\alpha,\beta,\psi)=1$.
\end{proof}

We are now ready to prove the following
\begin{theorem}
	Let $V$ be a vector space over $\korps$ whose dimension
	is $\geq 2$.
	\begin{enumerate}
		\item The Hochschild complex of the free algebra
		$\Tens V$ generated by $V$ is NOT formal in the
		$L_\infty$-sense.
		\item There is a $L_\infty$-structure on the
		Hochschild cohomology of $\Tens V$ whose Taylor coefficients $d_2$ and $d_3$ do not vanish, but
		$d_n=0$ for all $n\geq 4$, and there is a $L_\infty$-quis from the Hochschild cohomology
		(with respect to $d_2+d_3$) 
		to the Hochschild complex (with respect to the usual
		structure 
		$b[1]+[~,~]_G[1]$).
	\end{enumerate}
\end{theorem}
\begin{proof}
We look first at the finite-dimensional case where we shall
show the following auxiliary statement:

\vspace{0.5cm}

\begin{minipage}{10cm}
Let $\theta:\Lambda^2 \mathfrak{outder}\to \korps$ be a
linear map of tensor degree $0$ such that the
$\mathfrak{H}$-$3$-coboundary
$\delta_\mathfrak{H}\theta$ satisfies both eqs of eqn
(\ref{EqCompSigmaSpecial000And-101}). Then
$\delta_\mathfrak{H}\theta=0$.
\end{minipage} \hfill $(*)$

\vspace{0.5cm}

\noindent In the finite-dimensional case this will imply that the $3$-cocycle $\sigma$ is nontrivial since it is non-zero
showing nonformality whereas the second statement will then 
follow
from Theorem \ref{TFreeSigma}.

In order to prove the above auxiliary statement $(*)$, we first observe that $\theta$ has (up to obvious permutations)
only two surviving components, 
$\theta_{00}:\Lambda^2\Hom(V,V)\to \korps$, and 
$\theta_{-11}:V^*\otimes \Hom(V,V^{\otimes 2})\to \korps$
thanks to the fact that we can assume that $\theta$
is of tensor degree zero in order to have that 
$\delta_\mathfrak{H}\theta$ is. \\
The first of the eqs of
(\ref{EqCompSigmaSpecial000And-101}) shows that
$(\delta_{\mathfrak{gl}}\theta_{00})(A,B,C)=0$ for three
linear maps $A,B,C:V\to V$ where $\mathfrak{gl}$ is short
for the Lie algebra $\mathfrak{gl}(N,\korps)$. It is well-known
that the second scalar cohomology group of the Lie
algebra $\mathfrak{gl}(N,\korps)$ vanishes (where Whitehead's
Lemma for $\mathfrak{sl}(n,\korps)$ and the classical
Hochschild-Serre spectral sequence argument are used).
Hence there is a linear form 
$f_0:\Hom(V,V)\to\korps$ such that 
$\theta_{00}=\delta_{\mathfrak{gl}}f_0$. Upon trivially extending $f_0$ 
to a linear form (also denoted by $f_0$) of tensor degree
$0$ to the Lie algebra $\mathfrak{outder}$
we see that for the modified $2$-form $\theta'=\theta-\delta_{\mathfrak{H}}f_0$ the component
$\theta'_{00}$ vanishes whereas the coboundary remains
the same,
$\delta_{\mathfrak{H}}\theta'=\delta_{\mathfrak{H}}\theta$.\\
Next, with these modifications we look at the second equation
in (\ref{EqCompSigmaSpecial000And-101}): using
$\theta'_{00}=0$ we quickly obtain that $\theta'_{-11}$ is
invariant under the Lie algebra $\Hom(V,V)$. Since both
$\mathcal{H}^{0(1)}=\mathrm{Ker}S_1$ and $b'V$ are 
$\Hom(V,V)$-invariant complementary subspaces of $\Hom(V,V^{\otimes 2})$ we can extend $\theta'_{-11}$
to a $\Hom(V,V)$-invariant linear map from
$V\otimes\Hom(V,V^{\otimes 2})$ to $\korps$, hence as a
$\Hom(V,V)$-invariant element of $V^{\otimes 2}\otimes
V^{*\otimes 2}$. Using the Invariant Tensor Theorem (see e.g.~the book \cite[Thm.20.4, p.214]{KMS93} for a good account) we can conclude that $\theta'_{-11}$ is of the following form with $\lambda,\mu\in\korps$
(for all $\alpha\in V^*$ and $\rho\in \mathrm{Ker}S_2\subset\Hom(V,V^{\otimes 2})$)
\begin{equation}
      \theta'(\alpha,\rho) =
      \lambda\sum_{r,s=1}^N\alpha_r\rho_s^{rs}
      +\mu\sum_{r,s=1}^N\alpha_r\rho_s^{sr} 
      =\lambda\sum_{r,s=1}^N\alpha_r\rho_s^{rs}
\end{equation}
since $S_1(\rho)=\sum_{r,s=1}^N\rho_s^{rs}e_r=0$. We need the computation of the Lie bracket
in cohomology of $[\alpha,\psi]_H$ where
$\alpha\in V^*$ and $\psi\in \Hom(V,V^{\otimes 3})$:
\begin{eqnarray}
 \lefteqn{{[\alpha,\psi]_H}^{i_0i_1}_j } \nonumber \\
 & = & \sum_{r=1}^N\alpha_r\left(\psi_j^{ri_0i_1}
     +\psi_j^{i_0ri_1}+\psi_j^{i_0i_1r}\right)
     +\frac{1}{N-1}\sum_{r,s=1}^N
      \alpha_r\left(\psi_s^{rsi_0}\delta^{i_1}_j
                    -\psi_s^{rsi_1}\delta^{i_0}_j\right).
                    \nonumber \\
         &  &   \label{EqCompCohBracketAlphaPsiFreeAlg}
\end{eqnarray}
It is then straight-forward to see that  $\delta_{\mathfrak{H}}\theta'(\alpha,\beta,\psi)=0$
for all $\alpha,\beta\in V^*$ and 
$\psi\in \Hom(V,V^{\otimes 3})$. Since $\theta'$ is of tensor degree $0$, so is $\delta_{\mathfrak{H}}\theta'$, and
therefore the only possibly nonzero components are of tensor
degrees
$(0,0,0)$, $(-1,0,1)$, and $(-1,-1,2)$ which are all zero
whence $\delta_{\mathfrak{H}}\theta'=0$. This proves the
auxiliary statement and the
Theorem in the finite-dimensional case.
	
Suppose now that $V$ is a vector space of arbitrary, not necessarily finite dimension greater than $2$. Choose a subspace $W \subset V$ of finite dimension $N \in \nat$,
$N\geq 2$, and a complementary subspace $X\subset V$, \ie
$V=W\oplus X$. The inclusion $W\subset V$ induces an
inclusion of associative algebras $\iota:\Tens W\to \Tens V$.
Upon using the above
map $S$ for the finite-dimensional vector space $W$, there
is a complement $\mathcal{H}_W^{0}=\mathrm{Ker}S\subset \Hom(W,\Tens W)$ to $b'_W(\Tens W)$ where we have written
$b'_W$ for the adjoint representation with respect to $W$.
Let $Q_W:\Hom(W,\Tens W)\to \Tens W$ the above map
(\ref{EqDefQNFinDim}). For any linear map $\chi$ in
$\Hom(W,\Tens W)$ we define an extension $\tilde{\chi}$
in $\Hom(V,\Tens V)$ by $\tilde{\chi}(w)=\iota\big(\chi(w)\big)$ for all
$w\in W$, and $\tilde{\chi}(x)=\mathrm{ad}_{Q_W(\chi)}(x)=
b'_V\big(\iota(Q_W(\chi))\big)$ for all $x\in X$. 
We clearly get for all $a\in\Tens W$
\[
     \widetilde{b'_W(a)}= b'_V\big(\iota(a)\big),
\]
and the extension $\chi\mapsto \tilde{\chi}$ is injective.
We shall write $i:\Tens W \oplus \Hom(W,\Tens W)\to
\Tens V \oplus \Hom(V,\Tens V)$ for the injective linear map
$(a,\chi)\mapsto \big(\iota(a),\tilde{\chi}\big)$.
It clearly is a morphism of complexes 
$\big(\mathfrak{G}_{red}(\Tens W),b'_W\big)
\to \big(\mathfrak{G}_{red}(\Tens V),b'_V\big)$.
Next, we get the decomposition 
$\Tens V=\iota(\Tens W)\oplus \mathcal{I}$ where
$\mathcal{I}$ is the two-sided ideal of $\Tens V$ generated
by $X$. Note that for any $c\in \Tens V$ the adjoint representation $b'_V(c)=\mathrm{ad}_c$ preserves the subalgebra
$\iota(\Tens W)$ iff $c\in\iota(\Tens W)$, and it follows
\[
    \widetilde{\Hom(W,\Tens W)}\cap b'_V(\Tens V)
        =b'_V(\iota (\Tens W)).        
\]
Hence the subspace $\tilde{\mathcal{H}_W^0}$ trivially
intersects
the inner derivations $b'_V(\Tens V)=
b_V'(\iota(\Tens W))\oplus b_V'(\mathcal{I})$, hence we can
choose a tensor graded complement $\mathcal{H}^0_V$ of
$b'_V(\Tens V)$ in $\Hom(V,\Tens V)$ such that
$\widetilde{\mathcal{H}^0_W}\subset \mathcal{H}^0_V$.
We denote the projection $\Hom(V,\Tens V)\to \Tens V^+$
by $Q_V$.
It follows that
\[
     Q_V(\tilde{\chi})= \iota\big(Q_W(\chi)\big)
\]
hence the linear map $i$ also intertwines the corresponding
chain homotopies which we shall call $h_W$ and $h_V$, respectively. The first consequence is that the linear map
$\chi\mapsto \tilde{\chi}$, which is not a morphism of Lie
algebras, but it is one up to $b_V$-coboundaries, descends to
a Lie algebra injection $j:\mathfrak{outder}(\Tens W)\to
\mathfrak{outder}(\Tens V)$ (corresponding to the injection
$\widetilde{\mathcal{H}^0_W}\subset \mathcal{H}^0_V$).\\
With all these preparations we see that the characteristic
graded $3$-cocycles $\sigma_V$ (associated to the Hochschild complex of $\Tens V$) and $\sigma_W$ (associated to the Hochschild complex of $\Tens W$) are related by the map $j$,
viz.~
\[
    \sigma_V\big(j(\psi_1),j(\psi_2),j(\psi_3)\big)
    =   \sigma_W(\psi_1,\psi_2,\psi_3)
\]
for all $\psi_1,\psi_2,\psi_3\in \mathcal{H}_W^0$. If
$\sigma_V$ was exact, then by restriction $\sigma_W$
would also be exact in contradiction to the finite-dimensional
case.

\noindent This proves the Theorem.

\end{proof}

\begin{appendix}
	
\section{Filtered vector spaces}
  \label{App: Filtered vector spaces}

Most of the following material can be found e.g. in 
\cite[Ch.III]{Bou72} and \cite{NV79}.
Let $M$ be a vector space. Recall that a family of subspaces
$\big(F_r(M)\big)_{r\in\mathbb{Z}}$ of $M$ is called an
(ascending) \emph{filtration} 
if $F_r(M)\subset F_{r+1}(M)$ for any
integer $r$, and the pair $\left(M,\big(F_r(M)\big)_{r\in\mathbb{Z}}\right)$ is called a
\emph{filtered vector space}. Recall that the filtration is called \emph{exhaustive} if $\bigcup_{r\in\mathbb{Z}}F_r(M)=M$, 
\emph{separated} if $\bigcap_{r\in\mathbb{Z}}F_r(M)=\{0\}$, and
\emph{discrete} if there is an integer $r_0$ such that
$F_r(M)=\{0\}$ for all integers $r \leqslant r_0$. Every vector space
$M$ can be equipped with the \emph{trivial discrete filtration}
defined by $F_r(M)=\{0\}$ for all integers $r \leqslant -1$, and
$F_r(M)=M$ for all $r \geqslant 0$.
Moreover an exhaustive and separated filtration is well-known
to always give rise to a (topological) metric where the distance of two
elements $x,y$ of $M$ is defined by $2$ to the power
of the minimum of all
those integers $r$ such that $x-y\in F_r(M)$ if $x\neq y$,
and $0$ iff $x=y$. Hence a filtered vector space whose
filtration is exhaustive and separated is 
called \emph{complete} if the corresponding metric space is complete in the sense that every Cauchy sequence converges.
In such a situation a series $\sum_{n\in\mathbb{N}}x_n$
converges iff $x_n\to 0$, see e.g.~\cite[p.453]{Jac89}. Note that every filtered vector space whose filtration is exhaustive and discrete is complete.
Next, recall that for two filtered
vector spaces $\left(M,\big(F_r(M)\big)_{r\in\mathbb{Z}}\right)$
and $\left(M',\big(F_r(M')\big)_{r\in\mathbb{Z}}\right)$
a linear map $f:M\to M'$ is called of \emph{filtration degree} $m$
iff $f\big(F_r(M)\big)\subset F_{r+m}(M')$ for all integers $r$.
It follows that the space $H=\Hom \mathrm{filt} (M,M')$ of all linear maps of filtration degree $0$ is filtered
by declaring $F_r(H)=H$ for all $r \geqslant 0$ and for all
$r \leqslant -1$ $F_r(H)$ is the subspace of all linear maps of filtration degree $r \leqslant -1$. Note that 
the filtered vector space $H=\Hom \mathrm{filt} (M,M')$
is automatically complete if $\left(M,\big(F_r(M)\big)_{r\in\mathbb{Z}}\right)$
is exhaustive and separated and $\left(M',\big(F_r(M')\big)_{r\in\mathbb{Z}}\right)$ is complete.
Finally note that the tensor product of two filtered vector
spaces $M,M'$ is also filtered by $F_r(M\otimes M')=
\sum_{s\in\mathbb{Z}}F_s(M)\otimes F_{r-s}(M')$ for all integers
$r$.

\section{Graded Coalgebras}
  \label{App: Graded Coalgebras}
  
A lot of the following material can be found e.g. in
\cite[App.B]{Qui69}.
Recall that a graded vector space $C$ equipped with 
$\korps$-linear maps $\Delta_C:C\to C\otimes C$, 
$\varepsilon_C:C\to\korps$ and an element $\un_C$ 
of degree $0$
is called a \emph{graded coassociative counital coaugmented coalgebra} if
$(\Delta_C\otimes \mathrm{id}_C)\circ \Delta_C=
(\mathrm{id}_C\otimes\Delta_C)\circ \Delta_C$,
$(\varepsilon_C\otimes \mathrm{id}_C)\circ \Delta_C
=\mathrm{id}_C=(\mathrm{id}_C\otimes\varepsilon_C)$, $\Delta_C(\un_C)=\un_C\otimes \un_C$ and $\varepsilon(\un_C)=1$. Recall \emph{Sweedler's notation}
$\Delta(c)=\sum_{(c)}c^{(1)}\otimes c^{(2)}$ which stands for
a nonunique finite sum with homogeneous elements
$c^{(1)}$ and $c^{(2)}$ in $C$. Every such coalgebra carries a
\emph{canonical filtration} $\big(F_r(C)\big)_{r\in_\mathbb{Z}}$ defined
by $F_r(C)=\{0\}$ for all integers $r \leqslant -1$,
$F_0(C)=\korps\un_C$, and recursively
$F_{r+1}(C)=\{c\in C~|~
\Delta_C(c)-c\otimes \un_C-\un_C\otimes c\in F_r(C)\otimes 
F_r(C)\}$.
The maps $\Delta_C$ and $\varepsilon_C$ are filtration preserving.
A graded coassociative counital coaugmented coalgebra is called
a \emph{connected coalgebra} if the canonical filtration
is exhaustive. Most  of the graded coalgebras we shall encounter in this paper are \emph{graded cocommutative},
\ie $\tau\circ \Delta_C=\Delta_C$. Recall that a 
\emph{morphism
of graded connected coalgebras} $\Phi:C\to C'$ is a 
$\korps$-linear map of degree $0$ satisfying
$(\Phi\otimes \Phi)\circ \Delta_C=\Delta_{C'}\circ \Phi$,
$\varepsilon_{C'}\circ \Phi=\varepsilon_C$, and $\Phi(\un_C)=
\un_{C'}$. They are automatically filtration preserving. Moreover, a $\korps$-linear homogeneous map $d:C\to C'$ is called a \emph{graded coderivation of graded counital
coalgebras along the morphism $\Phi:C\to C'$} iff
$\Delta_{C'}\circ d= (d\otimes \Phi+\Phi\otimes d)\circ 
\Delta_C$. It follows that $\varepsilon_{C'}\circ d=0$.
For any graded associative unital algebra $(A,\mu_A, \un_A)$
and any  graded counital coassociative coalgebra
$(C,\Delta_C,\varepsilon_C)$ the \emph{convolution multiplication} $\ast$ on $\Hom(C,A)$ defined by
$\phi\ast \psi=\mu_A\circ (\phi\otimes \psi)\circ \Delta_C$
is a graded associative multiplication on $\Hom(C,A)$
with unit element $\un_A\varepsilon_C$. The convolution turns out to be very useful to express combinatorial formulas
in the graded symmetric bialgebra.

\section{The Perturbation Lemma}
\label{App: The Perturbation Lemma}

This Appendix is based on work of \cite{Bou72}, \cite{Hueb10},
\cite{Hueb11}, \cite{M10}:
\begin{defi}
	A \emph{(homotopy) contraction} consists of two chain complexes $(U,b_U)$ and $(V,b_V)$ (the differentials having degree $+1$) together with chain maps $i : U \to V$, $p : V \to U$, \ie
	\begin{subequations} \label{eq_contraction}
		\begin{gather}
		b_V \circ i = i \circ b_U, \qquad b_U \circ p = p \circ b_V \label{contraction_eq0}\\
		\intertext{
			and a map $h : V \to V$ of degree $-1$ such that
		}
		p \circ i = id_U \label{contraction_eq1} \\
		id_V - i \circ p = b_V \circ h + h \circ b_V \label{contraction_eq2} \\
		h^2 = 0, \qquad h \circ i = 0, \qquad p \circ h = 0. \label{contraction_sideqs}
		\end{gather}
	\end{subequations}
	Then $p$ is a surjection called the \emph{projection}, $i$ is an injection called the \emph{inclusion} and $h$ is an \emph{homotopy} between $id_V$ and $i \circ p$. We sum up equations \eqref{eq_contraction} with the diagram
	
	\begin{equation}\label{DiagramHomotopyContraction}
	\begin{tikzpicture}[baseline=(current bounding box.183)]
	\matrix (m) [matrix of math nodes,nodes in empty cells,column sep=1em,text height=1.5ex,text depth=0.25ex]
	{(U,b_U) & & (V,b_V) & \\};
	\path[right hook->]
	([yshift=-5pt]m-1-1.north east) edge node [above] {$i$} ([yshift=-5pt]m-1-3.north west);
	\path[->>]
	([yshift=5pt]m-1-3.south west) edge node [below] {$p$} ([yshift=5pt]m-1-1.south east);
	\draw[->] (m-1-4.north) arc (120:-120:2ex);
	\draw (m-1-4) ++(2em,0em) node {$h$};
	\end{tikzpicture}.
	\end{equation}
\end{defi}

\begin{remark}\label{RemarksContraction} \hfill
	\begin{enumerate}
		\item Condition \eqref{contraction_eq2} implies that the cohomologies of $U$ and $V$ are isomorphic.
		In the important particular case where the differential
		$b_U$ vanishes $U$ is isomorphic to the cohomology of $V$.
		\item Denoting by $[f,g] \coloneqq f \circ g - (-1)^{|f| |g|} g \circ f$ the graded commutator of two maps, this equation \eqref{contraction_eq2} also rewrites $id_V - i \circ p = [b_V,h]$. Our sign conventions are such
		thath both $i\circ p$ and $P:=[b_V,h]$ are idempotent
		linear maps. Note that $V$ decomposes in the direct sum of two subcomplexes, the kernel $V_U$ of $P$ (isomorphic to $(U,b_U)$,
		 and the image $V_\mathrm{acyc}$ of $P$ which is acyclic.
		\item \label{RemForcingSideconditions}
		Equations \eqref{contraction_sideqs} are called \emph{side conditions}. In case there is a homotopy contraction only satisfying eqs \eqref{contraction_eq0},
		\eqref{contraction_eq1}, and \eqref{contraction_eq2},
		it is straight-forward to see that the 
		`polynomially' modified homotopy
		\[
		  h'=[b_V,h]\circ h\circ b_V\circ  h\circ b_V\circ  h\circ b_V\circ  h\circ  [b_V,h]
		\]
		will satisfy all equations of \eqref{eq_contraction}
		including the side conditions 
		\eqref{contraction_sideqs}. Having a homotopy satisfying
		the side conditions is equivalent to specifying
		a vector space complement to the coboundaries in
		$V_\mathrm{acyc}$: $h$ will be zero on $V_U$ and on that complement and equal to the inverse of
		the restriction of the differential to that complement.
	 \item\label{RemSubcomplexLeadsToContraction} Since we are working in vector spaces, there is
	  a well-known important converse statement: if, for the two above complexes, 
	  there is just an injective chain map $i:(U,b_U)\to 
	  (V,b_V)$ 
	  inducing an isomorphism in cohomology, then there is
	  a surjective chain map $p:U\leftarrow V$ and a homotopy
	  $h: V\to V$ satisfying the conditions for a homotopy
	  contraction \eqref{contraction_eq2}: indeed 
	  by some straight-forward linear algebra it can be seen that it suffices to take a vector space complement $W$ to the subspace $i(U)+b_V(V)$ of $V$, then $V_\mathrm{acyc}=W\oplus b_V(W)$
	  will define an acyclic subcomplex of $(V,b_V)$ complementary to $i(U)$, which serves as a kernel of
	  an obvious surjective chain map $p:U\leftarrow V$. The chain homotopy $h$ is constructed as in Remark
	  \ref{RemarksContraction}, \ref{RemForcingSideconditions}.	  
	\end{enumerate}
\end{remark}

\begin{defi}
	A \emph{perturbation} of the differential $b_V$ is a morphism $\delta_V : V \to V$ of degree $+1$ such that $(b_V+\delta_V)^2 = 0 \Leftrightarrow \delta_V^2 + [\delta_V,b_V] = 0$.
\end{defi}

\begin{lemma}[Perturbation Lemma] \label{Lem:Perturbation}
	Let be given a contraction \eqref{DiagramHomotopyContraction}
	such that both $U$ and $V$ carry exhaustive and separated
	filtrations with $V$ complete (see Appendix \ref{App: Filtered vector spaces} for definitions) and such that the linear
	maps $b_U,b_V,i,p$ and $h$ are of filtration degree $0$.
	Moreover, let $\delta_V:V\to V$ be a perturbation of $b_V$ and suppose that $\delta_V$ is of filtration degree $-1$. \\
	Then the linear maps $(id_V + h\circ \delta_V)$ and
	$(id_V + \delta_V\circ h)$ from $V$ to $V$ are invertible,
	and we define
	\begin{gather*}
	\begin{aligned}
	\tilde{\imath} &= (id_V + h\circ \delta_V)^{-1}\circ i 
	\\
	\tilde{p} &= p \circ (id_V + \delta_V \circ h)^{-1} 
	\end{aligned}
	\qquad
	\begin{aligned}
	\tilde{h} &= (id_V + h \circ \delta_V)^{-1}\circ h 
	\\
	\delta_U &= p \circ(id_V + \delta_V\circ h)^{-1}\circ \delta_V\circ i 
	.
	\end{aligned}
	\end{gather*}
	Then $\delta_U$ is a perturbation of $b_U$ of filtration degree $-1$, and the above maps define a new contraction
	\begin{equation*}
	\begin{tikzpicture}[baseline=(current bounding box.183)]
	\matrix (m) [matrix of math nodes,nodes in empty cells,column sep=1em,text height=1.5ex,text depth=0.25ex]
	{(U,b_U+\delta_U) & & (V,b_V+\delta_V) & \\};
	\path[right hook->]
	([yshift=-5pt]m-1-1.north east) edge node [above] {$\tilde{i}$} ([yshift=-5pt]m-1-3.north west);
	\path[->>]
	([yshift=5pt]m-1-3.south west) edge node [below] {$\tilde{p}$} ([yshift=5pt]m-1-1.south east);
	\draw[->] (m-1-4.north) arc (120:-120:2ex);
	\draw (m-1-4) ++(2em,0em) node {$\tilde{h}$};
	\end{tikzpicture}.
	\end{equation*}
\end{lemma}
The inverse of $id_V+\chi$ where $\chi : V \to V$ is a $\korps$-linear map of filtration
degree $-1$ is defined by the geometric series
$\sum_{k=0}^\infty (-\chi)^{\circ k}$ which converges.
The verification of the above identities is straight-forward, see e.g.~\cite{Br65}.

\end{appendix}

\bibliographystyle{amsalpha}

\nocite{*}
\bibliography{biblio-formality}

\end{document}